\documentclass[12pt]{amsart}
\usepackage{amsmath,amssymb,amsthm}
\usepackage{latexsym}
\usepackage{color,colordvi,graphicx}
\numberwithin{equation}{section}
\newtheorem{theorem}{Theorem}
\newtheorem{proposition}[theorem]{Proposition}
\newtheorem{lemma}[theorem]{Lemma}
\newtheorem{corollary}[theorem]{Corollary}
\newtheorem{conjecture}[theorem]{Conjecture}

\theoremstyle{remark}

\newtheorem{example}[theorem]{Example}
\newtheorem{remark}[theorem]{Remark}

\newcommand{\defcolor}[1]{\Blue{#1}}
\newcommand{\demph}[1]{\defcolor{{\sl #1}}}

\newcommand{\C}{{\mathbb C}}
\newcommand{\K}{{\mathbb K}}
\newcommand{\R}{{\mathbb R}}
\newcommand{\Z}{{\mathbb Z}}
\renewcommand{\P}{{\mathbb P}}

\newcommand{\calY}{{\mathcal Y}}
\newcommand{\calH}{{\mathcal H}}

\newcommand{\url}[1]{{\tt #1}}

\newcommand{\Fdot}{F_{\bullet}}
\newcommand{\blambda}{\boldsymbol{\lambda}}
\newcommand{\bt}{\boldsymbol{t}}
\newcommand{\bu}{\boldsymbol{u}}
\newcommand{\bmu}{\boldsymbol{[\mu]}}

\newcommand{\Gr}{\mbox{\rm Gr}}
\newcommand{\LG}{\mbox{\rm LG}}
\newcommand{\HG}{\mbox{\rm HG}}
\newcommand{\Sym}{\mbox{\it Sym}}
\newcommand{\Wr}{\mbox{\rm Wr}}

\renewcommand{\sp}{\mathfrak{sp}}
\newcommand{\Sp}{\mbox{\rm Sp}}
\newcommand{\SL}{\mbox{\rm SL}}

\newcommand{\PGL}{\mbox{PGL}(2,\C)}

\newcommand{\Span}{\mbox{\rm span}}

\DeclareMathOperator{\ord}{ord}

\newcommand{\be}{{\bf e}}
\newcommand{\bbf}{{\bf f}}
\newcommand{\twoheadlongrightarrow}{\relbar\joinrel\twoheadrightarrow}

\newcommand{\I}{\raisebox{-1pt}{\includegraphics{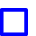}}}
\newcommand{\Is}{\includegraphics{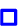}}
\newcommand{\TI}{\raisebox{-1pt}{\includegraphics{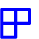}}}
\newcommand{\TT}{\raisebox{-1pt}{\includegraphics{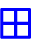}}}
\newcommand{\ThII}{\raisebox{-3pt}{\includegraphics{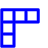}}}
\newcommand{\ThTI}{\raisebox{-3pt}{\includegraphics{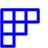}}}
\newcommand{\ThTT}{\raisebox{-7pt}{\includegraphics{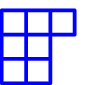}}}
\newcommand{\ThThT}{\raisebox{-3pt}{\includegraphics{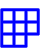}}}
\newcommand{\ThThTh}{\raisebox{-3pt}{\includegraphics{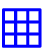}}}

\title[A congruence modulo four in real Schubert calculus]{A congruence modulo four in\\
      real Schubert calculus}

\author{Nickolas Hein}
\address{Nickolas Hein \\
         Department of Mathematics\\
         University of Nebraska at Kearney\\
         Kearney\\
         Nebraska \ 68849\\
         USA}
\email{heinnj@unk.edu}
\urladdr{http://www.unk.edu/academics/math/faculty/About\_Nickolas\_Hein/}
\author{Frank Sottile}
\address{Frank Sottile \\
         Department of Mathematics\\
         Texas A\&M University\\
         College Station\\
         Texas \ 77843\\
         USA}
\email{sottile@math.tamu.edu}
\urladdr{\url{http://www.math.tamu.edu/~sottile}}
\author{Igor Zelenko}
\address{Igor Zelenko\\
         Department of Mathematics\\
         Texas A\&M University\\
         College Station\\
         Texas \ 77843\\
         USA}
\email{zelenko@math.tamu.edu}
\urladdr{\url{http://www.math.tamu.edu/~zelenko}}
\thanks{Research of Sottile and Hein supported in part by NSF grant DMS-1001615}

\begin{document}

\begin{abstract}
 We establish a congruence modulo four in the real Schubert calculus on the
 Grassmannian of $m$-planes in $2m$-space.
 This congruence holds for fibers of the Wronski map and a generalization to what we call
 symmetric Schubert problems.
 This strengthens the usual congruence modulo two for numbers of real solutions to
 geometric problems.
 It also gives examples of geometric problems given by fibers of a map whose
 topological
 degree is zero but where each fiber contains real points.
\end{abstract}

\maketitle

%
\section*{Introduction}

The number of real solutions to a system of real equations is congruent
modulo two
to the number of complex solutions, for the simple reason that complex conjugation is an
involution which acts freely on the nonreal solutions.
We establish an additional congruence modulo four for
certain Schubert problems on the Grassmannian of $m$-planes in the space of polynomials of
degree at most $2m{-}1$, for $m>2$.
This congruence modulo four was originally observed in a computational experiment when
$m=3$ involving the Wronski map.
The reason for this congruence is that a natural symplectic structure on this space of
polynomials induces an additional geometric involution
on this Grassmannian which commutes with the Wronski map.
This key result (Lemma~\ref{L:commute}) is generalized in a sequel to this
paper~\cite{SA}.

This second involution commutes with complex conjugation, and the group they generate
consists of the identity and three involutions whose fixed points are, respectively, the
real Grassmannian, the Lagrangian Grassmannian, and a twisted real form of the Grassmannian,
which we call the Hermitian Grassmannian.
The common fixed point locus of this group is the real Lagrangian Grassmannian, and the
congruence modulo four is a consequence of it not forming a hypersurface in the real
Grassmannian when $m>2$.
This congruence is a general fact (Lemma~\ref{L:simple}) concerning fibers of a real map
that has a second involution,
whose
fixed point locus has codimension at least two.
It also applies to what we call symmetric Schubert problems
that
have this codimension
condition on their Lagrangian locus.
While we are unable to characterize which symmetric Schubert problems enjoy this
condition, we are able to
establish this condition for a large class of
symmetric Schubert problems.

Let \defcolor{$\K$} be a field which will either be the real numbers, \defcolor{$\R$}, or
the complex numbers, \defcolor{$\C$}.
We write \defcolor{$\K_d[t]$} for the $d{+}1$ dimensional vector space of univariate polynomials
of degree at most $d$ with coefficients from $\K$.
Given $f_1,\dotsc,f_m\in\K_{m{+}p{-}1}[t]$, their Wronskian is the determinant
 \[
   \det
    \left(\begin{matrix}
      f_1        & f_2        &\dotsb& f_m   \\
      f_1'       & f_2'       &\dotsb& f_m'  \\
     \vdots      &\vdots      &\ddots&\vdots \\
      f_1^{(m-1)}& f_2^{(m-1)}&\dotsb& f_m^{(m-1)}
    \end{matrix}\right)\ ,
 \]
which is a polynomial of degree at most $mp$.

Replacing $f_1,\dotsc,f_m$ by polynomials $g_1,\dotsc,g_m$ with the same linear span
will change their Wronskian by a constant, which is the determinant of the matrix
expressing the $g_i$ in terms of the $f_j$.
Thus the Wronskian is a well-defined map
\[
    \defcolor{\Wr}\ \colon\ \Gr(m,\K_{m+p-1}[t])\ \longrightarrow\
        \Gr(1,\K_{mp}[t])\ =\ \P(\K_{mp}[t])\,,
\]
where $\Gr(k,V)$ is the Grassmannian of $k$-dimensional subspaces of the $\K$-vector space
$V$.
Both $\Gr(m,\K_{m+p-1}[t])$ and $\P(\K_{mp}[t])$ are algebraic manifolds of dimension
$mp$, and $\Wr$ is a finite map.

When $\K=\C$, the degree of the Wronski map, which is the number of
points in a fiber above a regular value, was shown by Eisenbud and Harris~\cite{EH} (based
on earlier work of Schubert~\cite{Sch1886c}) to be
 \[
   \defcolor{\#^G_{m,p}}\ =\
      \frac{(mp)!\cdot 1!2!\dotsb(p{-}1)!}{m! (m{+}1)!\dotsb(m{+}p{-}1)!}\;.
  \leqno{(1)}
 \]
(Schubert determined the degree, and Eisenbud and Harris proved finiteness.)
In fact, this is the number of points in every fiber, if we count a point weighted by the 
algebraic multiplicity of the scheme-theoretic fiber at that point.
In all of our results, the bounds and congruences hold when the points are
counted with this multiplicity.

The real version of this \demph{inverse Wronski problem}, namely
studying the real
subspaces of polynomials in the fiber above a real polynomial $\Phi\in\R_{mp}[t]$, has been
the subject of recent interest.
This began with the conjecture of Boris Shapiro and Michael Shapiro ({\it circa} 1994), who
conjectured that if $\Phi$ had all of its roots real, then every ({\it a priori}) complex
vector space in the fiber $\Wr^{-1}(\Phi)$ would be real.
Significant evidence for this conjecture, both theoretical and computational, was found
in~\cite{So00b}.
Eremenko and Gabrielov~\cite{EG02} proved the conjecture when $\min\{m,p\}=2$, and it was
proved for all $m,p$ by Mukhin, Tarasov, and Varchenko~\cite{mtv1,mtv2}.
In their study of this conjecture, Eremenko and Gabrielov~\cite{EG01} calculated the
topological degree of the real Wronski map, obtaining a nonzero lower bound for the
number of real subspaces in the fiber above a general real polynomial of degree $mp$, when
$m{+}p$ is odd.
When $m{+}p$ is even, the topological degree is zero.
When both $m$ and $p$ are even, they gave a real polynomial of degree $mp$
with no real preimages under the Wronski map~\cite{EG03}.
When both $m$ and $p$ are odd it is not known if there is a nontrivial
lower bound on the number of real subspaces in the fibers of the Wronski map.
When $m=p$, we establish a congruence modulo four on the number of real points in a fiber
of the Wronski map.

\begin{theorem}\label{Th:Wronski}
 Suppose that $m=p$ and $m\geq 3$.
 For every $\Phi(t)\in\R_{m^2}[t]$, the number of real points in the fiber 
 $\Wr^{-1}(\Phi(t))$ is congruent to $\#^G_{m,m}$, modulo four, where each point is
 counted with its algebraic multiplicity.
\end{theorem}

Since $\#^G_{3,3}=42$, which is congruent to 2 modulo four, we obtain the following corollary.

\begin{corollary}\label{cor33}
 When $m=p=3$ in the fiber $\Wr^{-1}(\Phi(t))$ of a real polynomial $\Phi(t)$ of degree
 nine will contain either two simple real points or a real multiple point.
\end{corollary}

Since the topological degree of the Wronski map $\Wr\colon\Gr(3,\R_5[t])\to\P(\R_9[t])$ is
zero, this corollary shows that the lower bound can be larger than the topological degree.
The lower bound of
two from Corollary~\ref{cor33} is
attained because
we have examples of
polynomials $\Phi(t)$ of degree nine having only two real three-dimensional subspaces of
degree five polynomials in $\Wr^{-1}(\Phi(t))$.
Table~\ref{T:42} shows the result of a computing $1,000,000$ fibers of this Wronski map,
 \begin{table}[htb]
 %
 %
 \caption{Fibers of the Wronski map.}\label{T:42}

  \begin{tabular}{|c||c|c|c|c|c|c|c|c|c|c|c|c|}\hline
   Num. real&0&2&4&6&8&10&12&14&16&18&20&22\\\hline\hline
   Frequency&0&66380&0&310667&0&208721&0&51774&0&34524&0&45940\\\hline
  \end{tabular}\medskip

  \begin{tabular}{|c||c|c|c|c|c|c|c|c|c|c||c|}\hline
   Num. real&24&26&28&30&32&34&36&38&40&42&Total\\\hline\hline
   Frequency&0&8560&0&17881&0&6632&0&771&0&248150&1000000\\\hline
  \end{tabular}

 \end{table}
a computation that
consumed 391 gigaHertz-days.
The columns are labeled by the possible numbers of real solutions and each cell
in the second row
records
how many computed instances had that number of real solutions.
This paper originated in our desire to understand the result of this computation.

Many of the symmetric Schubert problems treated in Section~\ref{S:symmetric} also have a
lower bound of two on their number of real solutions, coming from the congruence modulo
four.
In Example~\ref{Ex:degreeZero} we give a family of
Schubert
problems generalizing
that of Corollary~\ref{cor33} and Table~\ref{T:42}.
For the problems in this family the topological degree of the corresponding Wronski map is
zero, but
there are always at least two points in every fiber.
This family illustrates
that these topologically derived lower bounds may not be sharp.

There is now a growing body of examples of geometric problems which have lower bounds on
their numbers of real solutions.
This phenomenon occurs not only in the Schubert calculus~\cite{EG01}, but also in counting
rational curves on varieties~\cite{W,IKS03,IKS04} and lines on hypersurfaces of degree
$2n{-}1$ in $\P^n$~\cite{OT1,FK}.
While there are some general theoretical bases for some of these lower
bounds~\cite{OT2,SS}, this phenomena is far from being understood.
Two of us conducted a large computational experiment of related lower bounds in the
(not necessarily symmetric) Schubert calculus, which was reported on in~\cite{Lower}.

While a similar congruence modulo four on the number of real solutions was also observed
in~\cite{AH}, such congruences appear to be a new phenomenon.
A consequence of these congruences is that the sharp lower bound on the number of
real solutions to these problems is congruent modulo 4 to the number of complex
solutions. 
A similar congruence for lower bounds also occurs when
enumerating rational curves on some Del Pezzo surfaces.
Welschinger~\cite{W} defined an invariant that is a lower bound on the
number of real rational curves in a given divisor class interpolating real points.
Later, Mikhalkin showed that this Welschinger invariant is congruent modulo four to the
number of complex curves, for toric Del Pezzo surfaces, and
this was extended to the other Del Pezzo surfaces ($\P^2$ blown up at $a=4,5,6$ points) by
Itenberg, Kharlamov, and Shustin~\cite{IKS10,IKS12}. 
(For more, see the discussion in~\cite[\S 7.2]{IKS12}.)

This paper is organized as follows.
In Section~\ref{S:definitions} we present some basics on Schubert calculus and derive the
canonical symplectic form on $\K_{2m-1}[t]$.
We establish a framework for congruences modulo four in Section~\ref{S:lemma}.
In Section~\ref{S:mod4} we prove Theorem~\ref{Th:Wronski}, and in
Section~\ref{S:symmetric} we extend this congruence to certain symmetric Schubert
problems.

%
\section{Definitions}\label{S:definitions}

All of our varieties and maps between varieties are defined over the
real numbers.
That is, they are complex varieties equipped with an antiholomorphic involution which we
call complex conjugation,
and the maps commute with the conjugation.
We will often write $X$ when we intend its set of complex points, $X(\C)$.
Let $X(\R)$ be the set of points of $X(\C)$ that
are fixed under complex conjugation.
Write $\Z_2$ for the group $\Z/2\Z$ with two elements and $[n]$ for the set
$\{1,2,\dotsc,n\}$ where $n$ is a positive integer.
We write \defcolor{$V^*$} for the linear dual of a vector space $V$.
Let $f\colon X\to Y$ be a map between varieties of the same dimension with $Y$ smooth.
A point $x\in X$ is a critical point of the map $f$ if the differential of
$f$ at $x$ is not an isomorphism of Zariski tangent spaces.

%
\subsection{Schubert Calculus}
Let $V$ be a vector space over $\K$ of dimension $m{+}p$ where $m,p$ are positive integers.
We write \defcolor{$\Gr(m,V)$} for the Grassmannian of $m$-dimensional linear
subspaces of $V$.
This equivalently parametrizes $m$-dimensional quotients of $V^*$.
This Grassmannian is a manifold of dimension $mp$ and is a homogeneous space for the
special linear group $SL(V)$.

A \demph{flag} is a sequence
$\Fdot\colon F_1\subset F_2\subset \dotsb \subset F_{m+p}=V$ of linear subspaces
of
$V$ with $\dim F_i=i$.
A \demph{partition} $\lambda$ is a weakly decreasing sequence of integers
$\lambda\colon p\geq\lambda_1\geq\dotsb\geq\lambda_m\geq 0$.
A flag $\Fdot$ and a partition $\lambda$ together determine a \demph{Schubert variety}
 \begin{equation}\label{Eq:SchubertVariety}
   \defcolor{X_\lambda\Fdot}\ :=\
    \{ H\in\Gr(m,V)\mid \dim H\cap F_{p+i-\lambda_i}\geq i\,,\quad i=1,\dotsc,m\}\,.
 \end{equation}
This has codimension $\defcolor{|\lambda|}:=\lambda_1+\dotsb+\lambda_m$ in $\Gr(m,V)$.
When $\lambda=(1,0,\dotsc,0)$ (written $\I$), the Schubert
variety is
 \[
   X_{\Is}\Fdot\ =\
    \{H\in\Gr(m,V)\mid \dim H\cap F_p\geq 1\}\,.
 \]
That is, those $H$ which meet $F_p$ nontrivially.

A \demph{Schubert problem} is a list
$\defcolor{\blambda}=(\lambda^1,\dotsc,\lambda^n)$ of partitions
where $|\lambda^1|+\dotsb+|\lambda^n|=mp$.
Given a Schubert problem $\blambda$ and general flags $\Fdot^1,\dotsc,\Fdot^n$, the
intersection of Schubert varieties
\[
   X_{\lambda^1}\Fdot^1\cap
   X_{\lambda^2}\Fdot^2\cap\dotsb\cap
   X_{\lambda^n}\Fdot^n
\]
is transverse~\cite{Kl74}.
The number \defcolor{$d(\blambda)$} of
complex
points in this intersection does not
depend upon the choice of general flags and may be computed using algorithms from the
Schubert calculus.
Our concern here is not in computing this number, but in congruences satisfied by numbers
of real solutions, for some Schubert problems and special choices of flags.

Suppose that $p=m$ and that $V$ is equipped with a symplectic form, i.e.\ a nondegenerate
alternating  form, denoted by $\langle\cdot,\cdot\rangle$.
The symplectic group \defcolor{$\Sp(V)$} is the subgroup of $\SL(V)$
consisting of linear transformations that
preserve this form
$\langle\cdot,\cdot\rangle$,
\[
   \Sp(V)\ =\ \{g\in\SL(V)\mid \langle gv,gw \rangle = \langle v,w\rangle\
                 \forall v,w\in V\}\,.
\]
A subspace $H\in\Gr(m,V)$ is \demph{Lagrangian} if $\langle H,H\rangle\equiv 0$.
The subset \defcolor{$\LG(V)$} of $\Gr(m,V)$ consisting of Lagrangian subspaces forms a
manifold of dimension $\binom{m+1}{2}$ and is a homogeneous space for the symplectic group
$\Sp(V)$.

The Lagrangian Grassmannian also has Schubert varieties
(see~\cite[Ch.~III]{FP98} for more details).
These require \demph{isotropic flags}, which are flags where
the subspace $F_i$ is the annihilator of $F_{2m-i}$ in that $\langle
F_i,F_{2m-i}\rangle\equiv 0$.
In particular, $F_m$ is Lagrangian.
We also need symmetric partitions, which we now explain.
A partition $\lambda\colon m\geq \lambda_1\geq\dotsb\geq\lambda_m\geq 0$ may be
represented by its Young diagram, which is an array of boxes with $\lambda_i$ boxes in row
$i$.
For example,
\[
   (2,1)\ \longleftrightarrow\ \raisebox{-3.5pt}{\includegraphics{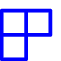}}\,,
   \qquad
   (3,2,2)\ \longleftrightarrow\ \ThTT\,,
   \qquad\mbox{and}\qquad
   (4,2,1,1)\ \longleftrightarrow\ \raisebox{-10.5pt}{\includegraphics{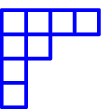}}\,.
\]
A partition is \demph{symmetric} if its Young diagram is symmetric about its main diagonal.
The partitions $(2,1)$ and $(4,2,1,1)$ are symmetric,
while
$(3,2,2)$ is not symmetric.

A symmetric partition $\lambda$ and an isotropic flag $\Fdot$ together determine a
Schubert variety of $\LG(V)$,
$\defcolor{Y_\lambda\Fdot}$, which is equal to
$X_\lambda\Fdot\cap \LG(V)$, so that
\[
   Y_\lambda\Fdot\ =\
    \{ H\in \LG(V)\mid \dim H\cap F_{m+i-\lambda_i}\geq i\,,\quad i=1,\dotsc,m\}\,.
\]
Its codimension in $\LG(V)$ is
\[
   \defcolor{\|\lambda\|}\ =\ \tfrac{1}{2}( |\lambda| + \ell(\lambda))\,,
\]
where $\ell(\lambda)$ is the number of boxes in the Young diagram of $\lambda$
that
lie on its main diagonal, $\ell(\lambda)=\max\{i\mid i\leq\lambda_i\}$.
For example,
\[
   \ell(\I)\ =\
   \ell(\TI)\ =\ 1
  \qquad\mbox{and}\qquad
   \ell(\TT)\ =\
   \ell(\ThTI)\ =\ 2\,.
\]
We compare this to the alternative indexing set by strict partitions $\kappa$,
which are strictly decreasing sequences of positive integers
$\kappa\colon m\geq\kappa_1>\dotsb>\kappa_k>0$.
Such a sequence is obtained from a symmetric partition $\lambda$ as the
subsequence of positive numbers in the decreasing sequence
$\lambda_1>\lambda_2-1>\dotsb>\lambda_k-k+1$.
The diagram of a strict partition is obtained by removing the boxes below the diagonal
from the diagram of a symmetric partition.
For example,
\[
   (5,3,2,1,1)\ \leftrightarrow\ \raisebox{-14pt}{\includegraphics{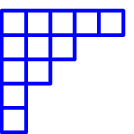}}
   \quad\rightsquigarrow\quad
   \raisebox{-14pt}{\includegraphics{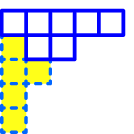}}
   \quad\rightsquigarrow\quad
   \raisebox{-3.5pt}{\includegraphics{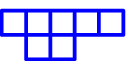}}
   \ \leftrightarrow\ (5>2)\,.
\]
One value of this alternative indexing is that $\|\lambda\|$ is the number of boxes in the
corresponding strict partition.

Given general isotropic flags $\Fdot^1,\dotsc,\Fdot^n$ and symmetric partitions
$\lambda^1,\dotsc,\lambda^n$, the intersection  of Schubert varieties
 \begin{equation}\label{Eq:Lagr-Int}
    Y_{\lambda^1}\Fdot^1\cap
    Y_{\lambda^2}\Fdot^2\cap\dotsb\cap
    Y_{\lambda^n}\Fdot^n
 \end{equation}
is generically transverse~\cite{Kl74}.
When $n=2$, we only need $\Fdot^1$ and $\Fdot^2$ to be in
\demph{linear general position}, $F^1_i\cap F^2_{2m{-}i}=\{0\}$, for $i=1,\dotsc,2m{-}1$.
In particular, the intersection~\eqref{Eq:Lagr-Int} is either empty or every component has
dimension
\[
   \binom{m+1}{2}-\|\lambda^1\|-\|\lambda^2\|-\dotsb-\|\lambda^n\|\,.
\]
This intersection consists of the Lagrangian subspaces
that belong to the intersection of Schubert varieties in the Grassmannian $\Gr(m,V)$,
\[
   \bigl(X_{\lambda^1}\Fdot^1\cap
    X_{\lambda^2}\Fdot^2\cap\dotsb\cap
    X_{\lambda^n}\Fdot^n\bigr) \cap \LG(V)\,.
\]

%
\subsection{Canonical symplectic form on $\K_{2m-1}[t]$}\label{S:alternating}

We follow the discussion of apolarity in \S1 of~\cite{O}.
Let $U$ be a vector space over $\K$ and $r$ a nonnegative integer.
Write \defcolor{$S^rU$} for the $r$-th symmetric power of $U$.
Its elements are degree $r$ homogeneous polynomials on \defcolor{$U^*$}, the vector space
dual to $U$, and thus $r$-forms on $\P(U)$, the projective space of hyperplanes of $U$.
The dual vector space to $S^rU$ is $S^rU^*$, whose elements act as differential operators
of degree $r$ on the homogeneous polynomials of degree $r$ representing the elements of 
$S^rU$. 

When $\dim U=2$, the exterior product leads to a symplectic form on
$U$, $\defcolor{\langle u,v\rangle}=u\wedge v$, which is well-defined up to a scalar
(corresponding to an identification of $\wedge^2 U$ with $\K$).
This induces a nondegenerate form \defcolor{$\langle\cdot,\cdot\rangle$} on
$S^r U$ which is well-defined up to a scalar multiple.
It is symmetric when $r$ is even and alternating when $r$ is odd.
Indeed, let $s,t$ span $U$ with $\langle s,t\rangle = 1$ and suppose that
$u=(u_0s+u_1t)^r$ and $v=(v_0s+v_1t)^r$.
Then a direct calculation gives $\langle u,v\rangle = (u_0v_1-v_0u_1)^r$, so that
$\langle u,v\rangle = (-1)^r\langle v,u\rangle$.
Then the claim about the symmetry of the form follows as $S^rU$ is spanned by $r$-th
powers of linear forms.
This computation gives the following formula:
when $u=\sum_{i=0}^r u_is^{r-i}t^i/i!$ and $v=\sum_{i=0}^r v_is^{r-i}t^i/i!$, then, up to a
scalar we have
 \begin{equation}\label{Eq:OurForm}
  \langle u,v\rangle\ =\
   \sum_{i=0}^r (-1)^i u_i v_{r-i}\,.
 \end{equation}
Henceforth we dehomogenize, setting $s=1$ and identifying $S^rU$ with $\K_r[t]$ and work with
these coordinates for $\K_r[t]$.
We will restrict to the case when $r$ is odd, writing $r=2m{-}1$, and thus
$\K_{2m-1}[t]$ has a natural structure of a symplectic vector space.

This symplectic form on $\K_{2m-1}[t]$ is classical and can be derived in several
different ways.
One attractive derivation is a consequence of $\K_{2m-1}[t]$ being
the space of solutions of the self-adjoint linear differential equation $y^{(2m)}=0$.
In fact, many of the results of this paper may be generalized to Schubert calculus on
spaces of solutions of a self-adjoint linear differential equation~\cite{SA}.

%
\subsection{Lagrangian involution and the Hermitian Grassmannian}\label{S:involution}

Suppose that $V\simeq \C^{2m}$ is a symplectic vector space with symplectic form
$\langle\cdot,\cdot\rangle$.
This form induces an involution on $\Gr(m,V)$ whose set of fixed
points is the Lagrangian Grassmannian.
For a linear subspace $K$ of $V$,
define $\defcolor{K^\angle}$ to be
$\{ v\in V\mid \langle u,v\rangle=0\ \mbox{for all }u\in K\}$.
We have $(K^\angle)^\angle=K$ and $\dim K + \dim K^\angle=2m$.

By this dimension calculation, $\angle$ restricts to an involution on $\Gr(m,V)$,
which we call the \demph{Lagrangian involution}.
Since $H$ is Lagrangian if and only if $H^\angle=H$, the Lagrangian Grassmannian is the
set of fixed points of the Lagrangian involution.

The real points  (those fixed by complex conjugation,
$x\mapsto\overline{x}$) of $\Gr(m,V)$ and $\LG(V)$ are,
respectively, the real Grassmannian \defcolor{$\R\Gr(m,V)$} and the real Lagrangian
Grassmannian, \defcolor{$\R\LG(V)$}.
These real points correspond to real linear subspaces of $V$ and real Lagrangian subspaces
of $V$, respectively.

There is another distinguished type of linear subspace of $V$.
An $m$-dimensional linear subspace $H$ of $V$ is \demph{Hermitian} if
\[
    \overline{H}\ =\ H^\angle
   \qquad\mbox{equivalently}\qquad H\ =\ \overline{H}^\angle\,.
\]
The \demph{Hermitian Grassmannian}  $\defcolor{\HG(V)}\subset\Gr(m,V)$ is the set of all
Hermitian linear subspaces of $V$.
The map $H\mapsto\overline{H}^\angle$ is an anti-holomorphic involution which equips
$\Gr(m,V)$ with a second real structure whose real points constitute the Hermitian
Grassmannian.
Thus $\HG(V)$ is a real algebraic manifold of dimension $m^2$ whose
complexification is $\Gr(m,V)$.
We have the following diagram of inclusions.
\[
  \begin{picture}(185,100)(5,0)
                            \put(78,88){$\Gr(m,V)$}
                   \put(100,84){\line(0,-1){25}}
   \put(85,84){\line(-2,-1){50}}             \put(115,84){\line(2,-1){50}}
   \put(5,45){$\R\Gr(m,v)$} \put(85,45){$\LG(V)$}  \put(155,45){$\HG(V)$}
   \put(85,13){\line(-2,1){50}}             \put(115,13){\line(2,1){50}}
                   \put(100,13){\line(0,1){25}}
                            \put(82,0){$\R\LG(V)$}
    \end{picture}
\]

\begin{proposition}\label{P:matrices}
 Each of these five Grassmannians may be realized as
the
 smooth
 compactification of a space of
 matrices according to the following table.
\[
  \begin{tabular}{|l||c|c|c|c|c|}\hline
    {\rm Space of matrices}&$M_{m\times m}(\C)$&$M_{m\times m}(\R)$
        \raisebox{-1pt}{\rule{0pt}{12.5pt}}
      &$\Sym_{m}(\C)$&$\Sym_{m}(\R)$&$\calH_m$\\\hline\hline
    {\rm Grassmannian}&$\Gr(m,V)$&$\R\Gr(m,V)$&$\LG(V)$&$\R\LG(V)$
      &$\HG(V)$\rule{0pt}{12pt}\\\hline
  \end{tabular}
\]
 Here \defcolor{$M_{m\times m}(\K)$} is the set of all $m\times m$ matrices over $\K$,
 \defcolor{$\Sym_m(\K)$} is $m\times m$ symmetric matrices over $\K$,
 and \defcolor{$\calH_m$} is all $m\times m$ Hermitian matrices.
\end{proposition}

\begin{proof}
Let $V\simeq{\mathbb C}^{2m}$ have basis  $\bbf_0,\dotsc,\bbf_{2m-1}$ for which the
 symplectic form is
\[
   \langle \bbf_i\,,\,\bbf_j \rangle\ =\ \left\{\begin{array}{rcl}
        \delta_{i+m,j}&\ &\mbox{if }i<m\\
       -\delta_{i,j+m}&&\mbox{if }i\geq m \end{array}\right. \ .
\]
 Let $I_m$ be the $m\times m$ identity matrix.
 In this ordered basis, the association
\[
    X\ \longmapsto\ \mbox{row space} [I_m: X]\,,
\]
 defines a map from the space of $m\times m$ matrices $X$ to the Grassmannian $\Gr(m,V)$.
  This identifies the space of $m\times m$ matrices with the big cell of $\Gr(m,V)$.
  In this cell, points of the real Grassmannian correspond to real matrices, points of the
  Lagrangian Grassmannian to symmetric matrices, and points of the Hermitian Grassmannian
  to Hermitian matrices.
\end{proof}

\begin{remark}
  If we consider the symmetric form on $V=\C^{2m}$ given by
  $\langle \bbf_i,\bbf_j\rangle=\delta_{|i-j|,m}$, then the space of isotropic $m$-planes
  in $V$---called the \defcolor{orthogonal Grassmannian}---corresponds to skew-symmetric
  matrices, those with $X^T=-X$, and there is a similar
  \defcolor{skew-Hermitian Grassmannian} corresponding to skew-Hermitian matrices,
  those with $\overline{X}^T=-X$.
\end{remark}

%
\section{A very simple lemma}\label{S:lemma}
Suppose that $f\colon X\to Z$ is a proper dominant map between irreducible 
complex varieties of the same dimension with $Z$ smooth.
Then $f$ has a degree, $d$, which is the number of (complex) inverse images of a regular
value $z\in Z$.
Suppose that $X$, $Z$, and the map $f$ are all defined over the real numbers, $\R$.
When $z\in Z(\R)$, we have the congruence,
 \begin{equation}\label{Eq:trivial_congruence}
   \# f^{-1}(z)\cap X(\R)\ \equiv\ d \mod 2\,,
 \end{equation}
which holds as the group $\Z_2$ generated by complex conjugation acts freely on the
nonreal points of the fiber.
When the map $f$ is finite, this congruence~\eqref{Eq:trivial_congruence} holds for all
$z$, when multiplicities are taken into account.
By multiplicity, we mean the usual Hilbert-Samuel
multiplicity of a point in a zero-dimensional scheme.

There is an additional congruence on the number of
real points when there is an additional involution which satisfies a simple hypothesis.

\begin{lemma}\label{L:simple}
 Suppose that $f\colon X\to Z$ is a proper dominant map of varieties defined over $\R$
 with $Z$ smooth.
 Suppose that the variety $X$ has an involution $\angle\curvearrowright X$ written
 $x\mapsto x^\angle$ satisfying $f(x^\angle)=f(x)$ such that the image $f(X_\angle)$ in
 $Z$ of the set of fixed points $X_\angle$ has codimension at least $2$,
 in that $\dim f(X_\angle) +2 \leq \dim Z$.

 If $y,z\in Z(\R)$ belong to the same connected component of\/$Z(\R)$ and the fibers above 
 them are finite and contain no points of $X_\angle$, then
 \[
   \# f^{-1}(y)\cap X(\R)\ \equiv\ \# f^{-1}(z)\cap X(\R)\ \mod 4\,.
 \]
\end{lemma}

\begin{proof}
 By our assumption on $\dim X_\angle$ and $y,z$, it is possible to connect $y$ with $z$ by
 a path $\gamma\colon[0,1]\to Z(\R)$ which does not meet the set $f(X_\angle)(\R)$ of
 points $w\in Z(\R)$ whose fiber $f^{-1}(w)$ meets $X_\angle$.
 Furthermore, we may assume that the image of $\gamma$ contains at most finitely many
 critical values of $f$.
 Let $S\subset[0,1]$ be those parameter values $s$ for which $\gamma(s)$ is a critical
 value.

 Pulling back regular fibers of $f$ along $\gamma$ and then taking the closure in
 $X(\C)\times[0,1]$ gives a map $f_\gamma\colon X_\gamma\to[0,1]$ with finite fibers.
 At a point $t\in[0,1]$ for which the fiber of $f$ over $\gamma(t)$ is finite, this 
 fiber has the same number of points as the fiber of $X_\gamma$ over $t$, and they have
 the same number of real points.
 If we let $c$ denote complex conjugation, then the group
 $K=\{e,c,\angle,c\angle\}\simeq \Z_2\times\Z_2$  acts on the fibers of $f_\gamma$.
 By our assumption on $\gamma$ and $X_\angle$, there are no $\angle$-fixed points in
 $X_\gamma$.

 Over each interval $I$ in $[0,1]\smallsetminus S$, $X_\gamma$ consists of a collection of
 disjoint arcs which may come together at endpoints of $I$ lying in $S$.
 Each of the arcs either has all of its points real (fixed by $c$), all of its points
 fixed by $c\angle$, or has $K$ acting freely on all of its points.
 The last two types of arcs contain no real points, and the second cannot have a real
 endpoint as $X_\gamma$ has no $\angle$-fixed points.

 Let $s\in S$ be a critical value and $I$ an interval in $[0,1]\smallsetminus S$ with $s$
 as one endpoint.
 The multiplicity of a point $x\in f_\gamma^{-1}(s)$ is the number of arcs in $X_\gamma$
 over $I$ with endpoint at $x$.
 It follows that the number of real points in $f_\gamma^{-1}(s)$, counted with multiplicity, is
 at least the number of real arcs over $I$.
 Only the real singular points $x\in f_\gamma^{-1}(s)$ can contribute to the difference.
 The contribution from $x$ is the number of
 complex arcs in $X_\gamma$ over $I$ that come together at $x$, which is an even number.
 The same even number of complex arcs comes together at $x^\angle$, which implies that 
 the pair $\{x,x^\angle\}$ contributes a multiple of four to any change in multiplicity at
 the point $s$.
 This completes the proof.
\end{proof}

\begin{remark}\label{R:Hermitian}
 We could have argued symmetrically that the number of points fixed by $c\angle$ may only
 change in multiples of four.
 This follows from Lemma~\ref{L:simple} as the composition $c\angle$ of the two
 involutions $c$ and $\angle$ gives a second real structure on $X$.
\end{remark}

\begin{corollary}
 Let $f\colon X\to Z$ and $\angle$ be as above.
 Then every point of $X_\angle$ is a critical point of $f$.
\end{corollary}

In particular, this shows that the hypotheses on the points $y$ and $z$ in
Lemma~\ref{L:simple} are satisfied if they are not critical values of $f$.

\begin{proof}
 Given a point $x\in X_\angle$, consider a curve $C$ in $Z$ that contains $f(x)$, is smooth
 at $f(x)$, and only meets $f(X_\angle)$ at $f(x)$.
 Then $\angle$ acts fiberwise on the inverse image $f^{-1}(C)$ of $C$ in $X$.
 We have that $x$ is a fixed point and $\angle$ acts freely on other points in $f^{-1}(C)$
 lying in a neighborhood of $x$.
 This implies that at least two branches of $C$ come together at $x$ and proves the
 corollary.
\end{proof}

%
\section{A mod 4 congruence in the real Schubert calculus}\label{S:mod4}

Let $m\geq 2$ be an integer.
We work in a vector space $V\simeq\C^{2m}$ equipped with the ordered basis
$\be_0,\be_1,\dotsc,\be_{2m-1}$ whose dual space $V^*$
has dual basis
$\be_0^*,\dotsc,\be_{2m-1}^*$.
That is $(\be_i,\be_j^*)=\delta_{i,j}$ where \defcolor{$(\cdot,\cdot)$} is the pairing
between $V$ and $V^*$.
Let $\defcolor{g}\colon\C\to V$ be the rational normal curve
$g(t):=\sum_i \be_i t^i/i!$.
Evaluating an element $u=\sum_i u_i \be_i^*$ of $V^*$ on $g(t)$ gives a polynomial
$(g(t),u)=\sum_i u_i t^i/i!$, and this identifies $V^*$ with $\C_{2m-1}[t]$.
Under this identification, the canonical symplectic form  of
Subsection~\ref{S:alternating} becomes $\langle u,v\rangle= uJ^{-1}v$,where
$J^{-1}:=\sum_i(-1)^i \be_i\otimes\be_{2m-1-i}$ is an isomorphism $V^*\to V$.
Its inverse, $J\colon V\to V^*$, is $\sum_i (-1)^i\be_i^*\otimes\be_{2m-1-i}^*$, which
gives a symplectic form on $V$, $\langle p,q\rangle:=pJq$ that is dual to the form on
$V^*$.

Let $\defcolor{\eta}\in\mbox{End}(V)$ be the following tensor,
\[
    \eta\ :=\  \sum_{i=0}^{2m-2} \be_{i+1}\otimes\be_{i}^*\ .
\]
This is nilpotent, $\eta^{2m}=0$.
Since $\eta$ has trace zero, it lies in $\mathfrak{sl}_{2m}$, the Lie algebra of the
special linear group $\mbox{SL}(V)$.
Since $\eta^TJ+J\eta=0$, it also lies in
$\sp(V)$, the Lie algebra of the symplectic group
$\mbox{Sp}(V)$ (which is defined with respect to the form
$\langle\cdot,\cdot\rangle$ on $V$).

For $t\in\C$, define $\defcolor{F_m(t)}$ to be
 \begin{equation}\label{Eq:fmt}
    e^{\eta t}\Span\{\be_0,\dotsc,\be_{m-1}\}\,.
 \end{equation}
Since $\Span\{\be_0,\dotsc,\be_{m-1}\}$ is Lagrangian and as $\eta\in\sp(V)$, we have
$e^{\eta t}\in\Sp(V)$, this~\eqref{Eq:fmt} defines a curve in $\LG(m)\subset\Gr(m,V)$.
Taking the limit as $t\to\infty$ gives
$F_m(\infty)=\Span\{\be_{m},\dotsc,\be_{2m-1}\}$, and defines $F_m(t)$ for $t\in\P^1$.

For $H\in\Gr(m,V)$ there will be $m^2$ points $t$ of $\P^1$ (counted with
multiplicity) for which $H$ meets $F_m(t)$ nontrivially.
(We explain this below.)
That is, points $t$ such that $H\in X_{\Is}\Fdot(t)$,
where $\Fdot(t)$ is a flag extending $F_m(t)$.
Conversely, given $m^2$ points $t_1,\dotsc,t_{m^2}$ of $\P^1$, it is a problem in
enumerative geometry to ask how many $H\in\Gr(m,V)$
meet each linear subspace $F_m(t_1),\dotsc,F_m(t_{m^2})$ nontrivially.
By work of Schubert~\cite{Sch1886c} and of Eisenbud and Harris~\cite{EH}, there
are $\#^G_{m,m}$~(1)
such planes $W$, counted with multiplicity.
Eisenbud and Harris' contribution
is that despite the flags $\Fdot(t_i)$ not being in general position, the corresponding
intersection of Schubert varieties is zero-dimensional.

We may also pose a version of this enumerative problem on the Lagrangian Grassmannian.
Given generic points $t_1,\dotsc,t_{\binom{m+1}{2}}$ of $\P^1$, how many Lagrangian subspaces
$W\in\LG(V)$
meet each of $F_m(t_1),\dotsc,F_m(t_{\binom{m+1}{2}})$ nontrivially?
By results of~\cite[Cor.~3.6]{Hi82} (for the degree) and~\cite{So00c} (for finiteness)
there are
\[
   \defcolor{\#^\angle_m}\ :=
     2^{\binom{m}{2}}\frac{\binom{m{+}1}{2}!\cdot 1!\dotsc(m{-}1)!}{1!3!\dotsb(2m{-}1)!}
\]
such Lagrangian subspaces, counted with multiplicity.
More specifically, there exists a choice of
$t_1,\dotsc,t_{\binom{m+1}{2}}\in\R\P^1$ for which there are finitely many such Lagrangian
subspaces, and none of them are real.

The first problem in the Schubert calculus on $\Gr(m,V)$ may be recast in the language of
Section~\ref{S:lemma}.
Let $H\in\Gr(m,V)$ be a $m$-dimensional linear subspace of $V$.
Its annihilator, $H^\perp\subset V^*$, also has dimension $m$.
Let $f_1,\dotsc,f_m$ be a basis for $H^\perp$, which we consider to be polynomials in
$\C_{2m-1}[t]$.
Their Wronskian is the determinant
\[
   \defcolor{\Wr}(f_1,\dotsc,f_m)\ :=\ \det
     \left(\begin{array}{cccc}
      f_1(t)&f_2(t)&\dotsb&f_m(t)\\
      f'_1(t)&f'_2(t)&\dotsb&f'_m(t)\\
      \vdots&\vdots&\ddots&\vdots\\
      f_1^{(m-1)}(t)&f_2^{(m-1)}(t)&\dotsb&f_m^{(m-1)}(t)
     \end{array}\right)\ ,
\]
which is a polynomial of degree at most $m^2$.
This Wronskian depends on the subspace $H$ up to multiplication by nonzero scalars
(coming from different bases for $H^\perp$), and thus gives a well-defined element of
the projective space  $\P(\C_{m^2}[t])\simeq\P^{m^2}$.

This defines the Wronski map
 \begin{equation}\label{Eq:WronskiMap}
   \Wr\ \colon\ \Gr(m,V)\ \twoheadlongrightarrow\ \P(\C_{m^2}[t])\,.
 \end{equation}
To relate the Wronski map to the Schubert calculus in $\Gr(m,V)$ first note that the
rational normal curve $g(t)$ is equal to $e^{\eta t}.\be_0$.
Moreover, given a point $u=\sum_i u_i \be_i^*\in V^*$, its evaluation,
$(e^{\eta t}.\be_i,u)$, on $e^{\eta t}.\be_i$ is the $i$th derivative of the polynomial
$(e^{\eta t}.\be_0, u)=(g(t),u)$ corresponding to $u$.
As above, let $H\in\Gr(m,V)$ and suppose that $f_1,\dotsc,f_m\in \C_{2m-1}[t]$ are a basis for
$H^\perp$.
As shown in~\cite[p.\ 123]{IHP}, the zeroes of the Wronskian
$\Wr(f_1,\dotsc,f_m)$ are exactly those numbers $s$ where $H\cap F_m(s)\neq\{0\}$.

We restate and prove Theorem~\ref{Th:Wronski}.\medskip

\noindent{\bf Theorem~\ref{Th:Wronski}.}
 {\it
  Suppose that $m\geq 3$.
  If\/ $\Phi$ is any real polynomial of degree at most $m^2$, then the number of real
  subspaces in $\Gr(m,\R_{2m-1}[t])$ whose Wronskian is proportional to $\Phi$ is congruent
  to $\#^G_{m,m}$ modulo four.

  Equivalently, given any subset $T=\{t_1,\dotsc,t_{m^2}\}$ of\/ $\P^1$ which is stable under
  complex conjugation, $\overline{T}=T$, the number of real subspaces $H\in \Gr(m,V)$
  whose complexifications meet each of $F_m(t_1),\dotsc,F_m(t_{m^2})$ nontrivially is congruent
  to $\#^G_{m,m}$ modulo four, where each point is counted with its algebraic
  multiplicity.}\medskip 

We state the key lemma of this paper.

\begin{lemma}\label{L:commute}
  The Wronski map commutes with the Lagrangian involution on $\Gr(m,V)$.
  That is,
  for $H\in\Gr(m,V)$,
\[
   \Wr(H^\angle)\ =\ \Wr(H)\,.
\]
\end{lemma}

\begin{proof}[Proof of Theorem~$\ref{Th:Wronski}$.]
 The Wronski map~\eqref{Eq:WronskiMap} is a
 finite map from $\Gr(m,V)$ to $\P(\C_{m^2}[t])$~\cite{EH}, both of which are irreducible
 smooth varieties. 
 (This implies that it is proper and dominant.)
 Since $\LG(V)=\Gr(m,V)_\angle$ and $\dim\LG(V)=\frac{1}{2}(m^2+m)$, we see that the
 image $\Wr(\LG(V))$ in $\P(C_{m^2}[t])$ has codimension at least two when $m\geq 3$.
 By Lemma~\ref{L:commute} the involution $\angle$ on $\Gr(m,V)$ commutes with $\Wr$.
 Since the set of real points of $\P(\C_{m^2}[t])$ (which is $\P(\R_{m^2}[t])$) are
 connected, Lemma~\ref{L:simple} implies that the number of real points in any two fibers
 above real polynomials are congruent modulo four.
 The congruence to $\#^G_{m,m}$ follows as there is a real polynomial $\Phi(t)$ with
 $\#^G_{m,m}$ real points in $\Wr^{-1}(\Phi)$~\cite{So99}.
\end{proof}

\begin{proof}[Proof of Lemma~$\ref{L:commute}$.]
 By continuity, it suffices to show this for $H$ whose
 Wronskian has only simple roots, as this set is dense in $\Gr(m,V)$.
 Since $\eta\in\sp(V)$, we have $e^{\eta t}\in\Sp(V)$.
 As $F_m(t)=e^{\eta t}.F_m(0)$ and $F_m(0)$ is Lagrangian, we see that if $H\in\Gr(m,V)$
 meets $F_m(t)$ nontrivially, then so does $H^\angle$.
 This implies that $\Wr(H)=\Wr(H^\angle)$.
\end{proof}

%
\section{A congruence modulo four in the symmetric Schubert calculus}\label{S:symmetric}

We retain the notation and definitions of Section~\ref{S:mod4} and
extend Theorem~\ref{Th:Wronski} to more general Schubert problems.

We defined the Lagrangian involution $\angle$ in Subsection~\ref{S:involution}.
For a linear subspace $K$ of $V$, recall that $K^\angle$ is its annihilator with
respect to the form $\langle\cdot,\cdot\rangle$.
If $\Fdot$ is a flag, then
\[
   \defcolor{\Fdot^\angle}\ \colon\
     (F_{2m-1})^\angle\ \subset\ (F_{2m-2})^\angle\ \subset\ \dotsb\
      \subset (F_2)^\angle\ \subset\ (F_1)^\angle\ \subset\ V
\]
is also a flag.
Furthermore, $\Fdot=\Fdot^\angle$ if and only if $\Fdot$ is isotropic.

For a partition $\lambda$, its \demph{transpose}, \defcolor{$\lambda^\angle$}, is
given by the matrix transpose of Young diagrams,
\[
   \ThTT^\angle\ =\ \;
   \raisebox{-7pt}{\includegraphics{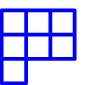}}
  \qquad\mbox{and}\qquad
   \raisebox{-7.5pt}{\includegraphics{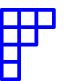}}^\angle\ =\ \;
   \raisebox{-5pt}{\includegraphics{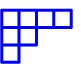}}\ .
\]
A partition $\lambda$ is symmetric if it equals it transpose, $\lambda=\lambda^\angle$.

\begin{lemma}\label{L:lagr-action}
 Let $\lambda$
 be a partition
 with $m\geq \lambda_1$ and $\lambda_{m+1}=0$
 and $\Fdot$ be a flag.
 Then
\[
  \angle \left(X_\lambda\Fdot\right)\ =\ X_{\lambda^\angle}\Fdot^\angle\,.
\]
\end{lemma}

\begin{proof}
 Write \defcolor{$\binom{[2m]}{m}$} for the set of all subsets of $[2m]$ of cardinality $m$.
 To a partition $\lambda$, we associate an element of $\binom{[2m]}{m}$,
\[
   a(\lambda)\ :=\ \{m+1-\lambda_1, m+2-\lambda_2,\dotsc, m+m-\lambda_m\}\,.
\]
 We define an involution on  $\binom{[2m]}{m}$,
\[
  \binom{[2m]}{m}\ni \alpha\ \longmapsto\
   \defcolor{\alpha^\angle}\ :=\ \{2m+1-i\mid i\not\in\alpha\}\,.
\]
 It is a pleasing combinatorial exercise to show that
\[
   a(\lambda)^\angle\ =\ a(\lambda^\angle)\,.
\]

 To prove the lemma, we use the following reformulation of the
 definition~\eqref{Eq:SchubertVariety} of a Schubert variety:
 \[
   X_\lambda\Fdot\ =\ \{H\in\Gr(m,V)\mid \dim H\cap F_i\geq
     \#(a(\lambda)\cap[i]),\quad i=1,\dotsc,2m\}\,.
\]
 We have the following chain of equivalences for $\alpha\in\binom{[2m]}{m}$ and flags
 $\Fdot$.
 \begin{multline*}
   \dim H\cap F_i\ \geq\ \#(\alpha\cap[i])\\
  \Leftrightarrow\
   \makebox[330pt][l]{$\dim\Span\{ H,F_i\}\ \leq\ m+i-\#(\alpha\cap[i])\ =\
    i+\#(\alpha\cap\{i{+}1,\dotsc,2m\})$}\\
  \Leftrightarrow\
    \makebox[372pt][l]{$\dim\Span\{ H,F_i\}^\angle\ \geq
      2m-i-\#(\alpha\cap\{i{+}1,\dotsc,2m\})\,.$}
 \end{multline*}
 But we have $(\Span\{ H,F_i\})^\angle=H^\angle\cap(F_i)^\angle=H^\angle\cap F^\angle_{2m-i}$
 and
 \[
   2m-i-\#(\alpha\cap\{i{+}1,\dotsc,2m\})\ =\ \#(\alpha^\angle\cap[2m{-}i])\,,
 \]
 which completes the proof.
\end{proof}

\begin{corollary}\label{Co:symmetric}
 Let $\Fdot$ be a isotropic flag and $\lambda$ a partition.
 Then $\angle(X_\lambda\Fdot)=X_\lambda\Fdot$ if and only if $\lambda$ is symmetric.
 When $\lambda$ is symmetric, the set of fixed points $(X_\lambda\Fdot)_\angle$ of
 $X_\lambda\Fdot$ is $Y_\lambda\Fdot$.
\end{corollary}

In Section~\ref{S:mod4} we defined a family $F_m(t)$ for $t\in\P^1$ of Lagrangian
subspaces of $V$.
We extend this to a family of isotropic flags.
For $j=1,\dotsc,2m$ and $t\in\C$ define
\[
  \defcolor{F_j}\ :=\ \Span\{\be_0,\dotsc,\be_{j-1}\}
  \qquad\mbox{and}\qquad
  \defcolor{F_j(t)}\ :=\ e^{\eta t} F_j\,.
\]
As $\langle F_j,F_{2m-j}\rangle=0$, the subspaces $F_j$ define an
isotropic flag, $\Fdot$.
Since $e^{\eta t}\in\Sp(V)$, the subspaces $F_j(t)$ also form an isotropic flag $\Fdot(t)$
with $\Fdot(0)=\Fdot$.
Taking the limit as $t\to\infty$
shows
\[
   F_j(\infty)\ =\ \Span\{\be_{2m-j},\dotsc,\be_{2m-1}\}\,,
\]
so that $\Fdot(\infty)$ is isotropic.
This family of flags has the property that if $a$ $b$ are distinct points of $\P^1$, then
$\Fdot(a)$ is in linear general position with respect to $\Fdot(b)$.
(By linear general position we mean that 
  $\dim F_i(a)\cap F_j(b)=\max\{0,i{+}j{-}2m\}$).\smallskip

Let $\blambda=(\lambda^1,\dotsc,\lambda^n)$ be a Schubert problem.
We consider a family of all corresponding intersections of Schubert varieties
given by flags $\Fdot(t)$ for $t\in\P^1$.
Write \defcolor{$(\P^1)^n_{\neq}$} for the set of $n$-tuples of distinct points.
Let $\defcolor{X_{\blambda}}\subset\Gr(m,V)\times(\P^1)^n$ be the closure of the incidence
variety
\[
   \defcolor{X^\circ_{\blambda}}\ :=\
   \{ (H,t_1,\dotsc,t_n) \mid (t_1,\dotsc,t_n)\in(\P^1)^n_{\neq} \mbox{ and }
         H\in X_{\lambda^i}\Fdot(t_i)\ i=1,\dotsc,n\}\,.
\]
Let $f\colon X_{\blambda}\to(\P^1)^n$ be the projection.
Since for $\bt=(t_1,\dotsc,t_n)\in(\P^1)^n_{\neq}$, we have
 \begin{equation}\label{Eq:total_family}
   f^{-1}(\bt)\ =\
   X_{\lambda^1}\Fdot(t_1)\,\cap\,
   X_{\lambda^2}\Fdot(t_2)\,\cap\, \dotsb\, \cap\,
   X_{\lambda^n}\Fdot(t_n)\,,
 \end{equation}
this family $f\colon X_{\blambda}\to(\P^1)^n$ contains all instances of the Schubert problem
$\blambda$ given by flags $\Fdot(t)$.

\begin{lemma}\label{L:Xblambda}
  For any Schubert problem $\blambda$ on $\Gr(m,V)$, the map
  $f\colon X_{\blambda}\to(\P^1)^n$ is finite and has degree $d(\blambda)$.
\end{lemma}

If the osculating flags in~\eqref{Eq:total_family} were in general position so that the
Schubert varieties met transversally, then $d(\blambda)$ is the number of points in the
intersection. 
The point of this Lemma is that even though  osculating flags are not general, we still
have that $d(\blambda)$ is the number of points in the
intersection~\eqref{Eq:total_family}. 

\begin{remark}
  Speyer~\cite{Speyer} constructed and studied a more refined
  compactification of $X^\circ_{\blambda}$.
  Since $\PGL$ acts on $\P^1$, and through the one-parameter subgroup $e^{\eta t}$
  it acts on $V$ and on $\Gr(m,V)$, we have that $\PGL$ acts on the family
  $X^\circ_{\blambda}\to(\P^1)^n_{\neq}$.
  The orbit space is a family of Schubert problems over
  $M_{0,n}:=(\P^1)^n_{\neq}/\PGL$, the open moduli space of $n$
  marked points on $\P^1$, which Speyer extended to a family over its
  compactification $\overline{M}_{0,n}$.
\end{remark}

\begin{proof}[Proof of Lemma~$\ref{L:Xblambda}$]
 Let us assume that $d(\blambda)\neq 0$ for otherwise $X_{\blambda}=\emptyset$.
 Consider the map $\chi\colon(\P^1)^n\to\P(\C_{m^2}[t])$ given by
\[
   \defcolor{\chi}\ \colon\ (t_1,\dotsc,t_n)\ \longmapsto\
     (t-t_1)^{|\lambda^1|} (t-t_2)^{|\lambda^2|}\dotsb  (t-t_n)^{|\lambda^n|}\,,
\]
 and let $\defcolor{X_*}\to (\P^1)^n$ be the pullback of the Wronski map
\[
   \Wr\ \colon\ \Gr(m,V)\ \longrightarrow\ \P(\C_{n^2}[t])
\]
 along the map $\chi$.

 Since the Wronski map is finite (it is a surjective map of smooth complete varieties with
 finite fibers), the map  $X_*\to (\P^1)^n$ is finite.
 Purbhoo studied the fibers of the Wronski map~\cite{Pu}.
 The set-theoretic fiber over a polynomial $g\in\P(\C_{n^2}[t])$ is
\[
    \bigcap_{\{s\colon g(s)=0\}} \bigcup_{\{\nu\colon |\nu|=\ord_s g\}} X_\nu \Fdot(s)\,.
\]
 In the scheme-theoretic fiber, the Schubert variety $X_\nu \Fdot(s)$ in this intersection
 is not reduced; it has multiplicity equal to the number of standard Young tableaux of shape
 $\nu$.
 Thus the set-theoretic fiber of $X_*$ over a point $(t_1,\dotsc,t_n)\in (\P^1)^n_{\neq}$ is
\[
    \bigcap_{i=1}^n \bigcup_{\{\nu\colon |\nu|=|\lambda^i|\}} X_\nu \Fdot(t_i)\,.
\]
 It follows that $X_{\blambda}$ is a union of irreducible components of $X_*$ and so
 $X_{\blambda}\to (\P^1)^n$ is finite.
\end{proof}

We call a Schubert problem $\blambda=(\lambda^1,\dotsc,\lambda^n)$ \demph{symmetric} if
each partition $\lambda^i$ is symmetric.
By Lemma~\ref{L:lagr-action}, if $\blambda$ is a symmetric Schubert problem, then the fibers
of $X_{\blambda}\to(\P^1)^n$  are preserved by the Lagrangian involution.
We would like to
show
that the number of real points in a fiber is congruent to $d(\blambda)$ modulo
four.
Unfortunately, $X_{\blambda}\to(\P^1)^n$ is not the right family for this.
If $(t_1,\dotsc,t_n)$ is a real point in $(\P^1)^n_{\neq}$, then each $t_i$ is
real and the Mukhin-Tarasov-Varchenko Theorem~\cite{mtv1,mtv2} states that the
intersection~\eqref{Eq:total_family} consists of $d(\blambda)$ real points, so that the
desired congruence to $d(\blambda)$ modulo four is not interesting.

The problem with the family $X_{\blambda}\to(\P^1)^n$ is that the points $t_1,\dotsc,t_n$
are distinguishable as they are labeled. 
Since $\overline{F_i(t)}=F_i(\overline{t})$, we have
 \begin{equation}\label{Eq:SVarConj}
    \overline{X_\lambda\Fdot(t)}\ :=\
    \{ \overline{H} \mid H\in X_\lambda\Fdot(t) \}\ =\
    X_\lambda\Fdot(\overline{t})\,,
 \end{equation}
and so there may be fibers of $X_{\blambda}\to(\P^1)^n$ which are real (stable under complex
conjugation), for which the corresponding ordered $n$-tuple $(t_1,\dotsc,t_n)$ is not
fixed by complex conjugation.
To remedy this, we replace $(\P^1)^n$ by a space in which the points $t_i$ and $t_j$ of
$\P^1$ are indistinguishable when $\lambda^i=\lambda^j$.

Our remedy uses an alternative representation of a symmetric Schubert problem which
records the frequency of the different partitions.
Suppose that $\{\mu^1,\dotsc,\mu^r\}$ are the distinct partitions appearing in a Schubert
problem $\blambda$ so that $\mu^i\neq\mu^j$ for $i\neq j$, and $\mu^i$ occurs $n_i>0$
times in $\blambda$.
In that case, we write $\defcolor{\bmu}=\{(\mu^1,n_1),\dotsc,(\mu^r,n_r)\}$ for
the Schubert problem.
For example,
\[
  \left\{
   \bigl(\,\TT,2\bigr)\,,\,
   \bigl(\,\TI,4\bigr)\,,\,
   \bigl(\,\I,5\bigr)
  \right\}
  \qquad\mbox{and}\qquad
  \left\{
   \bigl(\,\ThTI,1\bigr)\,,\,
   \bigl(\,\TI,5\bigr)\,,\,
   \bigl(\,\I,4\bigr)
  \right\}
\]
are symmetric Schubert problems on $\Gr(5,\C^{10})$.
We may write $\bmu=\defcolor{\bmu(\blambda)}$ to indicate that $\bmu$ and $\blambda$ are the
same Schubert problem and define $d(\bmu):=d(\blambda)$ when this occurs.

Given a symmetric Schubert problem in this form
$\bmu=\{(\mu^1,n_1),\dotsc,\{\mu^r,n_r)\}$, set
\[
   \defcolor{Z_{\bmu}}\ :=\
   \{ (u_1,u_2,\dotsc,u_r)\mid u_i\in\P(\C_{n_i}[t])\}
    \ \simeq\ \prod_{i=1}^r \P^{n_i}\,,
\]
and let \defcolor{$U_{\bmu}$} be the subset consisting of $(u_1,\dotsc,u_r)\in Z_{\bmu}$
where each polynomial $u_i$ is square-free and any two are relatively prime
(i.e.\ the roots of $u_1\dotsb u_r$ all have multiplicity one).
Let $\defcolor{X_{\bmu}}\subset \Gr(m,V)\times Z_{\bmu}$ be the closure of
the incidence correspondence
\[
   \{(H\,,\,u_1,\dotsc,u_r)\mid (u_1,\dotsc,u_r)\in U_{\bmu}
     \mbox{\ and\ }H\in X_{\mu^i}(s)\mbox{\ for\ }u_i(s)=0\,, i=1,\dotsc,r\}\,.
\]
Let $f\colon X_{\bmu}\to Z_{\bmu}$ be the projection.
Since for $\bu=(u_1,\dotsc,u_r)\in U_{\bmu}$, we have
 \begin{equation}\label{Eq:fibersofbmu}
     f^{-1}(\bu)\ =\
     \bigcap_{i=1}^r \ \bigcap_{\{s\colon u_i(s)=0\}} X_{\mu^i}\Fdot(s)\,,
 \end{equation}
the family $f\colon X_{\bmu}\to Z_{\bmu}$ is another family containing all instances of
the Schubert problem $\bmu$ given by flags $\Fdot(t)$.

By~\eqref{Eq:SVarConj} an intersection~\eqref{Eq:fibersofbmu} is stable under complex
conjugation exactly when the roots of $u_i$ are stable under complex conjugation.
That is, exactly when each $u_i$ is a real polynomial, so that
$(u_1,\dotsc,u_r)\in U_{\bmu}(\R)$.

\begin{lemma}\label{L:finite}
  The map $f\colon X_{\bmu}\to Z_{\bmu}$ is finite.
\end{lemma}

%
 This may be proved in nearly the same manner as the proof of
 the finiteness of $X_{\blambda}\to(\P^1)^n$.
 Consequently, we only indicate the differences.
 The main difference is that the map $\chi$ is replaced
 by the map $\varphi\colon Z_{\bmu}\to\P(\C_{n^2}[t])$ defined by
\[
    Z_{\bmu}\ \ni\ (u_1,\dotsc,u_r)\ \longmapsto\
     u_1^{|\mu^1|}\cdot u_2^{|\mu^2|}\dotsb u_r^{|\mu^r|}\ \in\ \P(\C_{m^2}[t])\,.
\]

When $\bmu=\bmu(\blambda)$ so that $\bmu$ and $\blambda$ are the same Schubert problem, the
map $\chi\colon(\P^1)^n\to\P(\C_{m^2}[t])$ factors through $\varphi$.
Define the map $\psi\colon(\P^1)^n\to Z_{\bmu}$ by
 \begin{equation}\label{Eq:psi}
   \defcolor{\psi}\ \colon\ (t_1,\dotsc,t_n)\ \longmapsto\
    \Bigl( \prod_{\{j\in[n] \mid \lambda^j=\mu^i\}} (t-t_j) \mid i=1,\dotsc,r\Bigr)\,.
 \end{equation}
Then $\chi=\varphi\circ\psi$, and we have the commutative diagram
\[
  \begin{picture}(191,53)(0,-1)

   \put(5,42){$X_{\blambda}$}   \put(74,42){$X_{\bmu}$} \put(143,42){$\Gr(m,V)$}
    \put( 30,45){\vector(1,0){40}}
    \put( 99,45){\vector(1,0){40}}

   \put( 11,36){\vector(0,-1){23}}\put(0,23){$f$}
   \put( 82,36){\vector(0,-1){23}}\put(71,23){$f$}
   \put(164,36){\vector(0,-1){23}}\put(166,23){$\Wr$}

   \put(0,0){$(\P^1)^n$}   \put(74,0){$Z_{\bmu}$} \put(143,0){$\P(\C_{m^2}[t])$}
    \put( 30,3){\vector(1,0){40}} \put( 47,7){\footnotesize$\psi$}
    \put( 99,3){\vector(1,0){40}} \put(116,7){\footnotesize$\phi$}
  \end{picture}
\]

The Lagrangian involution $\angle$ on $\Gr(m,V)$ induces an involution $\angle$ on
$\Gr(m,V)\times Z_{\bmu}$ which acts trivially on the second factor.
When $\bmu$ is symmetric,
this restricts to an involution $\angle$ on $X_{\bmu}$ with $f(H^\angle)=f(H)$,
by Corollary~\ref{Co:symmetric}.
In this context, the Theorem of Mukhin, Tarasov, and Varchenko implies that if
$\bu=(u_1,\dotsc,u_r)$ is a point of $U_{\bmu}$ in which the polynomials $u_i$ have distinct
real roots, then all points in the fiber $f^{-1}(\bu)$ consists of $d(\blambda)$ real points.
Thus we have the following corollary of
Lemma~\ref{L:simple} for $f\colon X_{\bmu}\to Z_{\bmu}$.

\begin{theorem}\label{Th:stuck}
 Suppose that $\bmu$ is a symmetric Schubert problem.
 If $f((X_{\bmu})_\angle)$ has codimension at least $2$ in $Z_{\bmu}$, then the number of
 real points in a fiber of $f$ over $u\in U_{\bmu}(\R)$ is congruent to $d(\blambda)$
 modulo four, where each point is counted with its algebraic multiplicity. 
\end{theorem}

Since the map $\psi\colon(\P^1)^n\to Z_{\bmu}$ is finite, the map $X_{\blambda}\to X_{\bmu}$ is
finite and so this condition on codimension  is equivalent to the
image of $(X_{\blambda})_\angle$ in $(\P^1)^n$ having dimension at most $n{-}2$.

Both aspects of Theorem~\ref{Th:stuck}, the condition on codimension and the congruence
modulo four, are quite subtle.
We will give several examples which serve to illustrate this subtlety.
Conjecture~\ref{C:Bertini} below gives a combinatorial condition which should imply
the condition on codimension and hence the congruence modulo four, but we cannot prove it
as we lack a Bertini-type theorem.
We instead offer a weaker condition that we can prove.
Recall that if $\lambda$ is a symmetric partition, then $\ell(\lambda)$ is the number of
boxes in the main diagonal of the Young diagram of $\lambda$.

\begin{theorem}\label{Th:simple_condition}
  Let $\blambda=(\lambda,\mu,\nu^1,\dotsc,\nu^n)$ be a symmetric Schubert problem in which
  either $\lambda\neq\mu$ or else $\lambda=\mu$ and there is some $i$ with
  $\lambda=\mu=\nu^i$.
  Set $\bmu=\bmu(\blambda)$.
  If
 \begin{equation}\label{Eq:simple_condition}
   2\ +\ \frac{1}{2}\left(|\nu^1|+\dotsb+|\nu^n|
     \ +\ m-\ell(\lambda)-\ell(\mu) \right)\ \leq\ n\,,
 \end{equation}
 then the number of real points in a fiber of $X_{\bmu}$ over a real point of $U_{\bmu}$ is
 congruent to $d(\blambda)$ modulo four, where each point is counted with its algebraic
 multiplicity. 
\end{theorem}

We defer the proof of Theorem~\ref{Th:simple_condition}, first giving a family of
  Schubert problems for which it gives lower bounds, and then another example of a
  Schubert problem to which it applies.

\begin{example}\label{Ex:degreeZero}
 The work of Eremenko and Gabrielov on lower bounds in the Schubert calculus~\cite{EG01}
 was generalized in Theorem 6.4 of~\cite{SS}, which applies to the Wronski map restricted
 to certain intersections of two Schubert varieties.
 In some cases when the topological degree is zero and
 Theorem~\ref{Th:simple_condition} applies
 the congruence modulo four implies that every fiber has at least two real
 points as in Corollary~\ref{cor33}.
 We explain this.

 Let $\lambda$ and $\mu$ be partitions encoding Schubert conditions on $\Gr(m,V)$.
 These define a \demph{skew partition}, \defcolor{$\lambda^c/\mu$}.
 Here, $\lambda^c$ is the partition complementary to $\lambda$
 in that $\lambda_i^c=m-\lambda_{m+1-i}$ for $i=1,\dotsc,m$ and $\lambda^c/\mu$ is
 the collection of boxes that remain in $\lambda^c$ after removing $\mu$.
 Implicit in this construction is that $\mu\subset\lambda^c$.
 Below are two examples, one when $m=4$ with $\lambda=\TI$ and $\mu=\ThII$,
 and the other when $m=5$ with $\lambda=\ThTI$ and $\mu=\ThThT$.
 The skew partition is unshaded.
 \begin{equation}\label{Eq:skew}
   \raisebox{-10.5pt}{\includegraphics{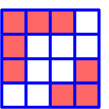}}
   \qquad\qquad
   \raisebox{-14pt}{\includegraphics{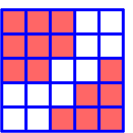}}
 \end{equation}

 Let $\defcolor{|\lambda^c/\mu|}:=m^2-|\lambda|-|\mu|$ be the number of boxes in the skew
 partition $\lambda^c/\mu$.
 Restricting the Wronski map to the intersection
 $\defcolor{\Omega_{\lambda,\mu}}:=\Omega_\lambda\Fdot(\infty)\cap\Omega_\mu\Fdot(0)$
 of Schubert varieties (and dividing by $t^{|\mu|}$) gives a finite map
 \begin{equation}\label{Eq:Wr_restricted}
   \Wr_{\lambda,\mu}\ \colon\  \Omega_{\lambda,\mu}\ \longrightarrow\
    \P(\C_{|\lambda^c/\mu|}[t])\,.
 \end{equation}
 Its degree is the number \defcolor{$f(\lambda^c/\mu)$} of
 standard Young tableaux of skew shape $\lambda^c/\mu$~\cite{Fu97}.
 The boxes in $\lambda^c/\mu$ form a poset in which a box is
 covered by it neighbors immediately to its right or below and $f(\lambda^c/\mu)$ is the
 number of linear extensions of this poset~\cite{St86}.

 The variety $\Omega_{\lambda,\mu}$ in the Pl\"ucker embedding of the Grassmannian
 $\Gr(m,V)$ is a variety whose coordinate ring is an algebra with straightening law on the
 poset $\lambda^c/\mu$, which is a distributive lattice.
 In these Pl\"ucker coordinates the map $\Wr_{\lambda,\mu}$ is what is called
 in~\cite[\S 6]{SS} a \demph{Wronski projection for $\lambda^c/\mu$ with constant sign}.
 Furthermore, the subset of $\Omega_{\lambda,\mu}$ where the minimally indexed Pl\"ucker
 coordinate $x_\mu$ does not vanish is an open subset of affine space and is therefore
 orientable.
 By Theorem~6.4 of~\cite{SS} the degree (called there the characteristic) of the Wronski map
 is the \demph{sign-imbalance of the poset $\lambda^c/\mu$}.

 This sign-imbalance is
 defined as follows.
 Fixing one linear extension of $\lambda^c/\mu$, all others are obtained from it by a
 permutation.
 The sign-imbalance of $\lambda^c/\mu$ is the absolute value of the sum of the signs of
 these permutations.

 When $\lambda$ and $\mu$ are symmetric, transposition induces an involution $T\mapsto
 T^\angle$ on  the set of Young tableaux/linear extensions.
 The permutations corresponding to $T$ and to $T^\angle$ differ by the product of
 transpositions, one for each pair of boxes in $\lambda^c/\mu$ that are interchanged by
 transposing.
 Thus when there is an odd number of such pairs, $T$ and $T^\angle$ contribute opposite signs
 to the sign-imbalance.
 Since $T\neq T^\angle$ when $|\lambda^c/\mu|>1$, this argument also shows that
 $f(\lambda^c/\mu)$ is even.
 We deduce the following lemma.

\begin{lemma}
   When a symmetric skew partition $\lambda^c/\mu$ has an odd number of boxes above its main
   diagonal, its sign-imbalance is zero.
\end{lemma}

 The skew partitions in~\eqref{Eq:skew} have three and five boxes above their main
 diagonal, respectively, and so both have sign-imbalance zero.
 For symmetric partitions $\lambda$ and
 $\mu$, we have a symmetric Schubert problem
 $\bmu:=\{(\lambda,1),(\mu,1),(\I,|\lambda^c/\mu|)\}$.
 Then $Z_{\bmu}$ is the product $ \P^1\times \P^1\times \P(\C_{|\lambda^c/\mu|}[t])$ and
 the restricted Wronski map~\eqref{Eq:Wr_restricted} is simply the restriction of
 the map $X_{\bmu}\to Z_{\bmu}$ to the set $\{(0,\infty)\}\times\P(\C_{|\lambda^c/\mu|}[t])$,
 and the whole family  $X_{\bmu}\to Z_{\bmu}$ is the closure of the
 $\PGL$-orbit of~\eqref{Eq:Wr_restricted}.
 Thus we lose nothing by considering the restricted Wronski
 map~\eqref{Eq:Wr_restricted} in place of $X_{\bmu}\to Z_{\bmu}$.

\begin{corollary}\label{Cor:degzero}
  Let $\lambda$ and $\mu$ be symmetric partitions.
  If $4+m\leq|\lambda^c/\mu|+\ell(\lambda)+\ell(\mu)$, $\lambda^c/\mu$ has an odd number
  of boxes above its main diagonal, and $f(\lambda^c/\mu)$ is congruent to two
  modulo four, then the degree of the real restricted Wronski map~\eqref{Eq:Wr_restricted}
  is zero, but its fiber over every real polynomial contains either two simple
  real points or a multiple point.
\end{corollary}

 The condition $4+m\leq|\lambda^c/\mu|+\ell(\lambda)+\ell(\mu)$ is the condition of
 Theorem~\ref{Th:simple_condition} for the Schubert problem
 $\{(\lambda,1),(\mu,1),(\I,|\lambda^c/\mu|)\}$ as $n=|\lambda^c/\mu|$.
 Thus Corollary~\ref{Cor:degzero} gives a class of geometric problems in which the lower
 bound given by the degree is not sharp.
 It generalizes Corollary~\ref{cor33} as the problem $\{(\I,9)\}$ on $\Gr(3,\C^6)$ with 42
 solutions satisfies the hypotheses of Corollary~\ref{Cor:degzero}.
 There are two such Schubert problems on $\Gr(4,\C^8)$,
\[
   \{(\ThThT,1),(\I,8)\}\mbox{ with 90 solutions}
   \quad\mbox{and}\quad
   \{(\ThII,1),(\TT,1),(\I,8)\}\mbox{ with 426 solutions.}
\]
 We give the numbers of problems satisfying the hypotheses of Corollary~\ref{Cor:degzero} 
 for small values of $m$.
\[
  \begin{tabular}{|c||c|c|c|c|c|}\hline
   $m$ &3&4&5&6&7\\\hline
   number &1&2&7&18&34\\\hline
  \end{tabular}
\]
The problem for the second skew tableau of~\eqref{Eq:skew} has $40,370$ solutions, and the
largest on $\Gr(7,\C^{14})$ has $\lambda=(5,4,2,2,1)$, $\mu=(6,5,2,2,2,1)$, and
$843,201,530$ solutions.
%
\end{example}

\begin{example}
 Consider the Schubert problem
 $\blambda=(\,\ThTI, \ThII,  \I,\I,\I,\I,\I)$
 on $\Gr(4,\C^8)$ with 40 solutions.
 Here $n=5$, $\lambda=\ThTI$,
 $\mu=\ThII$, and
 $\nu^1,\dotsc,\nu^5=\I$.
 Then~\eqref{Eq:simple_condition} is
 \[
    5\ \geq\ 2 + \frac{1}{2}(1+1+1+1+1\:+\: 4-2-1)
     \ =\ 5\,.
 \]
 Theorem~\ref{Th:simple_condition} implies that this Schubert problem exhibits a
 congruence modulo four.
 We have observed this in experimentation.
 Table~\ref{T:40} displays the result of a computation which took $30.6$ gigaHertz-days.
 \begin{table}[htb]
 %
 %
 \caption[Schubert problem with 40 solutions]{Schubert problem
   $\left\{
    \bigl(\,\I,5\bigr)\,,\,
    \bigl(\,\ThII,1\bigr)\,,\,
    \bigl(\,\ThTI,1\bigr)
   \right\}$ with 40 solutions.}\label{T:40}

  \begin{tabular}{|c||c|c|c|c|c|c|c|c|c|c|c|}\hline
   Num. real&0&2&4&6&8&10&12&14&16&18&20\\\hline\hline
   Frequency&181431&0&62673&0&48005&0&29422&0&8360&0&11506\\\hline
  \end{tabular}\medskip

  \begin{tabular}{|c||c|c|c|c|c|c|c|c|c|c||c|}\hline
   Num. real&22&24&26&28&30&32&34&36&38&40&Total\\\hline\hline
   Frequency&0&14137&0&6123&0&2040&0&9696&0&226607&600000\\\hline
  \end{tabular}

 \end{table}
 The columns are labeled by the possible numbers of real solutions, and each cell records how
 many computed instances had that number of real solutions.
\end{example}

\begin{proof}[Proof of Theorem~$\ref{Th:simple_condition}$.]
 Let $\blambda=(\lambda,\mu,\nu^1,\dotsc,\nu^n)$ be a symmetric Schubert problem for
 $\Gr(m,V)$.
 Recall that $f\colon X^\circ_{\blambda}\to(\P^1)^{2+n}_{\neq}$ is the family whose fiber over
 $(a,b,t_1,\dotsc,t_n)\in(\P^1)^{2+n}_{\neq}$ is
\[
    X_\lambda\Fdot(a)\,\cap\,
    X_\mu\Fdot(b)\,\cap\,
    \bigcap_{i=1}^n  X_{\nu^i}\Fdot(t_i)\,.
\]
 Consider its closure $\widetilde{X}_{\blambda}$ in
 $\Gr(m,V)\times(\P^1)^2_{\neq}\times(\P^1)^n$---the points $a$ and $b$ defining the first two
 Schubert varieties remain distinct, but any other pair may collide.

 The fiber of $\widetilde{X}_{\blambda}$ over a point
 $\bt:=(a,b,t_1,\dotsc,t_n)\in(\P^1)^{2+n}_{\neq}$ is a subset of
 \[
        X_\lambda\Fdot(a)\,\cap\, X_\mu\Fdot(b)\,.
 \]
 The fiber of $(\widetilde{X}_{\blambda})_\angle$ over the same point $\bt$ is a subset of
 $(X_\lambda\Fdot(a)\,\cap\, X_\mu\Fdot(b))_\angle$, which is
 \begin{equation}\label{Eq:Lagr_int}
       (X_\lambda\Fdot(a))_\angle \,\cap\, (X_\mu\Fdot(b))_\angle
     \ =\
        Y_\lambda\Fdot(a)\,\cap\, Y_\mu\Fdot(b)\,.
 \end{equation}
 As $a\neq b$, the flags $\Fdot(a)$ and $\Fdot(b)$ are in linear general position and so this
 intersection is generically transverse in $\LG(V)$.
 Thus~\eqref{Eq:Lagr_int} has dimension
 \begin{eqnarray}
   \dim \LG(V)-\|\lambda\|-\|\mu\| &=&\nonumber
   \frac{1}{2}(m^2+m)\ -\ \frac{1}{2}(|\lambda|+\ell(\lambda))
    \ -\ \frac{1}{2}(|\mu|+\ell(\mu))\\\nonumber
  &=& \frac{1}{2}(m^2+m-|\lambda|-|\mu|-\ell(\lambda)-\ell(\mu))\\
  \label{Eq:myDim}
  &=& \frac{1}{2}(|\nu^1|+\dotsb+|\nu^n|\ +\ m-\ell(\lambda)-\ell(\mu))\,,
\end{eqnarray}
 because $|\lambda|+|\mu|+|\nu^1|+\dotsb+|\nu^n|=m^2$ as $\blambda$ is a Schubert problem on
 $\Gr(m,V)$.

 Let $\defcolor{\calY_{\lambda,\mu}}\subset\LG(V)\times(\P^1)^2_{\neq}$ be
 the family of intersections~\eqref{Eq:Lagr_int}
\[
   \{(H,a,b)\: \mid\: H\in Y_\lambda\Fdot(a)\cap Y_\mu\Fdot(b)\}\,,
\]
 which has dimension
\[
   \dim \calY_{\lambda,\mu}\ =\
    2\ +\ \frac{1}{2}(|\nu^1|+\dotsb+|\nu^n|\ +\ m-\ell(\lambda)-\ell(\mu))\,.
\]
 Forgetting $(t_1,\dotsc,t_n)$ gives a map
 $\pi\colon(\widetilde{X}_{\blambda})_\angle \to \calY_{\lambda,\mu}$.
 Its fiber over a point $(H,a,b)$ is
\[
  \{(t_1,\dotsc,t_n)\:\mid\: H\in X_{\nu^i}\Fdot(t_i)\mbox{ for }i=1,\dotsc,n\}\,.
\]
 Since $H\in X_{\nu}\Fdot(t)$ implies that
 $H\in X_{\Is}\Fdot(t)$, and this second condition
 occurs for only finitely many $t\in\P^1$ (these are the zeroes of the Wronskian of $H$),
 the fiber $\pi^{-1}(H,a,b)$ is either empty or it is
a finite set.
 Thus
 \begin{equation}\label{Eq:dimension_estimate}
   \dim (\widetilde{X}_{\blambda})_\angle \ \leq\
     2\ +\ \frac{1}{2}(|\nu^1|+\dotsb+|\nu^n|\ +\ m-\ell(\lambda)-\ell(\mu))\,,
 \end{equation}
 and so
 the condition~\eqref{Eq:simple_condition} implies that
\[
  \dim  f((\widetilde{X}_{\blambda})_\angle)\ \leq\
  \dim (\widetilde{X}_{\blambda})_\angle\ \leq\
  n\ =\ \dim ((\P^1)^2_{\neq}\times(\P^1)^n)-2\,.
\]

 To complete the proof, let $\bmu:=\bmu(\blambda)$ and consider the map $\psi$ defined
 in~\eqref{Eq:psi}
\[
  \psi\ \colon\ (\P^1)^2_{\neq}\times(\P^1)^n\ \longrightarrow\ Z_{\bmu}\,.
\]
 To apply the argument in the proof of Lemma~\ref{L:simple} to the family
 $f\colon X_{\bmu}\to Z_{\bmu}$ requires that there exists a curve
 $\gamma\subset Z_{\bmu}(\R)$ connecting any two points $\bu,\bu'\in U_{\bmu}(\R)$
 such that $f^{-1}(\gamma)\cap (X_{\bmu})_\angle=\emptyset$.
 This is always possible if the codimension of $f((X_{\bmu})_\angle)$ is at least 2,
 $Z_{\bmu}$ is smooth, and $Z_{\bmu}(\R)$ is connected.

 From the first part of this proof, we have control over the points of $(X_{\bmu})_\angle$
 lying over the image of $\psi$.
 Thus we seek a curve $\gamma$ lying in the real points of the image of $\psi$.
 This image consists of polynomials $(u_1,\dotsc,u_r)\in Z_{\bmu}$ where the root $a$
 corresponding to $\lambda$ is distinct from the root $b$ corresponding to $\mu$.
 The curve $\gamma$ connecting $\bu,\bu'\in U_{\bmu}(\R)$ must be a curve of real
 polynomials where these roots remain distinct.
 This can always be done when either $\lambda\neq\mu$ (so that $a$ and $b$ are roots of
 different polynomials), or else $\lambda=\mu$ and there is some $i$ with $\mu=\nu^i$.
 In this second case, we require that the polynomial $u_j$ corresponding to this common
 partition has at least two distinct roots at every point of $\gamma$---which is possible,
 as $\deg u_j\geq 3$.
\end{proof}

\begin{example}\label{Ex:12}
 The condition of Theorem~\ref{Th:simple_condition} is sufficient, but by no means necessary
 for there to be a congruence modulo four.
 In particular, the estimate~\eqref{Eq:dimension_estimate} on the dimension of
 $(\widetilde{X}_{\blambda})_\angle$ could be improved.
 Consider the problem
 $\blambda=(\,\ThII, \TT,\TT,  \I,\I,\I)$
 on $\Gr(4,\C^8)$ with 12 solutions.
 For Condition~\eqref{Eq:simple_condition}, we have
 $n=4$,
 $\lambda=\ThII$,
 $\mu=\TT$, and
 $\nu^1,\dotsc,\nu^4$ equal to
 $\TT, \I,\I,\I$.
 Then~\eqref{Eq:simple_condition} becomes
\[
  4 \geq 2 + \frac{1}{2}(4+1+1+1\ +\ 4-1-2)\ =\ 6\,,
\]
 which does not hold.
 Nevertheless, we observed a congruence modulo four in this Schubert problem.
 Table~\ref{T:12} displays the results of a computation that consumed 150.8
 gigaHertz-days of computing.
 %
 %
 \begin{table}[htb]
  \caption[Schubert problem with 12 solutions]{Schubert problem $\left\{
    \bigl(\,\I,3\bigr)\,,\,
    \bigl(\,\TT,2\bigr)\,,\,
    \bigl(\,\ThII,1\bigr)
   \right\}$ with 12 solutions.}\label{T:12}
  \begin{tabular}{|c||c|c|c|c|c|c|c||c|}\hline
   Num. real&0&2&4&6&8&10&12&Total\\\hline\hline
   Frequency&214375&0&231018&0&61600&0&293007&800000\\\hline
  \end{tabular}
 \end{table}
 While we did not explicitly compute $\dim f((\widetilde{X}_{\blambda})_\angle)$, it is
 at most 3, and we expect it to be $2$, based on heuristic arguments that we give
 below. 
\end{example}

 Suppose that $\blambda$ is symmetric and $\bt=(t_1,\dotsc,t_n)\in(\P^1)^n_{\neq}$.
 Then the fiber of $X_{\blambda}$ over $\bt$ is the intersection~\eqref{Eq:total_family}.
 By Corollary~\ref{Co:symmetric}, the points of $(X_{\blambda})_\angle$ lying over $\bt$
 are
\[
   Y_{\lambda^1}\Fdot(t_1)\,\cap\,
   Y_{\lambda^2}\Fdot(t_2)\,\cap\, \dotsb\, \cap\,
   Y_{\lambda^n}\Fdot(t_n)\,,
\]
 which is a subscheme of $\LG(V)$.
 Since the codimension of $Y_{\lambda}\Fdot(t)$ in $\LG(V)$ is $\|\lambda\|$,
 it is reasonable to conjecture that the expected dimension of such an intersection
 gives the dimension of the image of $(X_{\blambda})_\angle$ in $(\P^1)^n$, which would then
 imply a congruence modulo four.
 We make a conjecture based on these observations.

\begin{conjecture}\label{C:Bertini}
  Suppose that $\blambda=(\lambda^1,\dotsc,\lambda^n)$ is a symmetric Schubert problem for
  $\Gr(m,V)$.
  If we have
 \begin{equation}\label{Eq:codimension}
    2\ \leq\ \|\lambda^1\|\ +\ \dotsb\ +\ \|\lambda^n\|
    \ -\ \dim\LG(V)\,,
 \end{equation}
  then  the number of real points in a fiber of $X_{\bmu}$ over a real point of $U_{\bmu}$ is
  congruent to $d(\blambda)$ modulo four, where $\bmu=\bmu(\blambda)$.
\end{conjecture}

\begin{lemma}\label{L:compare}
  The inequality~$\eqref{Eq:simple_condition}$ implies
  the inequality~$\eqref{Eq:codimension}$.
\end{lemma}

\begin{proof}
 Let $\blambda=(\lambda,\mu,\nu^1,\dotsc,\nu^n)$ be a symmetric Schubert problem on
 $\Gr(m,V)$.
 Rewrite the inequality~\eqref{Eq:simple_condition} as
\[
  2\ +\ \frac{1}{2}(m-\ell(\lambda)-\ell(\mu))\ \leq\
   n\ -\ \frac{1}{2}(|\nu^1|+\dotsb+|\nu^n|)\,.
\]
 Since $1\leq \|\nu\|=\frac{1}{2}(|\nu|+\ell(\nu))$, this implies that
%
\[    2\ +\ \frac{1}{2}(m-\ell(\lambda)-\ell(\mu))\ \leq\
  \frac{1}{2}(\ell(\nu^1)+\dotsb+\ell(\nu^n))\,.
\]
%
 As $\blambda$ is a Schubert problem on $\Gr(m,V)$, we have
 $m^2=|\lambda|+|\mu|+|\nu^1|+\dotsb+|\nu^n|$, and so this becomes
 \begin{eqnarray*}
   2& \leq&\frac{1}{2}\bigl(
     \ell(\lambda)+\ell(\mu)+\ell(\nu^1)+\dotsb+\ell(\nu^n)-m
      +|\lambda|+|\mu|+|\nu^1|+\dotsb+|\nu^n|-m^2 \bigr)\\
  &=&\|\lambda\|+\|\mu\|+\|\nu^1\|+\dotsb+\|\nu^n\| - \frac{1}{2}(m^2+m)\,,
 \end{eqnarray*}
 which is the inequality~\eqref{Eq:codimension} of Conjecture~\ref{C:Bertini} as
 $\dim \LG(V)=\frac{1}{2}(m^2+m)$
\end{proof}

\begin{remark}
 For the Schubert problem of Example~\ref{Ex:12}, we have
 \begin{multline*}
  \qquad
   \|\I\| +  \|\I\| +  \|\I\| +
   \|\TT\| +    \|\TT\| +
   \|\ThII\|
    \ -\ \dim\LG(\C^8)\\
  \ =\
 1+1+1+3+3+3-10\ =\ 2\,.\qquad
 \end{multline*}
 Thus the inequality~\eqref{Eq:codimension} holds, and so Conjecture~\ref{C:Bertini} predicts
 the congruence modulo four that we observed in Example~\ref{Ex:12}.

 In every example of a Schubert problem we have computed in which the
 inequality~\eqref{Eq:codimension} holds, we have observed this congruence  to $d(\blambda)$
 modulo four.

 The intuition behind Conjecture~\ref{C:Bertini} is the following.
 Let $\defcolor{\calY_{\lambda}}\subset\LG(V)\times\P^1$ be
 $\{(H,t)\mid H\in Y_\lambda\Fdot(t)\}$, which is the family over $\P^1$ whose fiber over
 $t\in\P^1$ is $Y_\lambda\Fdot(t)$.
 This has codimension $\|\lambda\|$ in $\LG(V)\times\P^1$.
 Let $\defcolor{\Delta}\colon\LG(V)\to\LG(V)^n$ be the diagonal map.
 Suppose that $\blambda=(\lambda^1,\dotsc,\lambda^n)$ is a symmetric Schubert problem.
 Then the fiber product of $\calY_{\lambda^1},\dotsc,\calY_{\lambda^n}$ over $\LG(V)$ is
 \begin{equation}\label{Eq:fiberProduct}
    \Bigl(\calY_{\lambda^1}\times\calY_{\lambda^2}\times\dotsb\times\calY_{\lambda^n}\Bigr)
    \;\bigcap\; \Bigl(\Delta\times 1_{(\P^1)^n}\bigl(\LG(V)\times(\P^1)^n\bigr)\Bigr)\,.
 \end{equation}
 The codimension of $\calY_{\lambda^1}\times\dotsb\times\calY_{\lambda^n}$ in
 $(\LG(V)\times\P^1)^n$ is $\|\lambda^1\|+\dotsb+\|\lambda^n\|$ and the dimension of
 $\LG(V)\times(\P^1)^n$ is $\frac{1}{2}(m^2+m) +n$.
 Thus the expected dimension of~\eqref{Eq:fiberProduct} is
\[
   \frac{1}{2}(m^2+m) +n\ -\
    \|\lambda^1\|-\|\lambda^2\|-\dotsb-\|\lambda^n\|\,,
\]
 which is the difference of $n=\dim(\P^1)^n$ and the number~\eqref{Eq:codimension}.
 Thus Conjecture~\ref{C:Bertini} and the observed congruence modulo four would follow
 from a Bertini-type theorem for the families $\calY_{\lambda^i}$ implying that the
 intersection~\eqref{Eq:fiberProduct} is proper.
\end{remark}

The inequality~\eqref{Eq:codimension} is not the final word on this congruence modulo four.
When it fails, there may or may not be a congruence modulo four.
We illustrate this with three examples.

The Schubert problem
 $\blambda=(\ThTI, \ThII, \TI, \I,\I)$
 on $\Gr(4,\C^8)$ has 14 solutions.
 We have
\[
   \|\I\| + \|\I\| +  \|\TI\| + \|\ThII\| +  \|\ThTI\|
    \ -\ \dim\LG(\C^8)
  \ =\
 1+1+2+3+4-10\ =\ 1\,,
\]
 so the inequality~\eqref{Eq:codimension} does not hold.
 Table~\ref{T:14} shows the result of computing 200,000 instances of this problem,
%
%
\begin{table}[htb]
 \caption[Schubert problem with 14 solutions]{
  Schubert problem $\left\{
   \bigl(\,\I,2\bigr)\,,\,
   \bigl(\,\TI,1\bigr)\,,\,
   \bigl(\,\ThII,1\bigr)\,,\,
   \bigl(\,\ThTI,1\bigr)
  \right\}$ with 14 solutions.}\label{T:14}

 \begin{tabular}{|c||c|c|c|c|c|c|c|c||c|}\hline
  Num. real&0&2&4&6&8&10&12&14&Total\\\hline\hline
  Frequency&38008&17926&14991&4152&6938&210&6038&111737&200000\\\hline
 \end{tabular}

\end{table}
which took $8.7$ gigaHertz-days.
We observed every possible number of real solutions and thus there is no congruence modulo
four for this problem.

The Schubert problem
 $\blambda=(\,\ThTI, \ThTI, \TI, \I)$
 on $\Gr(4,\C^8)$ has 8 solutions.
 We have
\[
   \|\I\| + \|\TI\| +  \|\ThTI\| + \|\ThTI\|
    \ -\ \dim\LG(\C^8)
  \ =\
 1+2+4+4-10\ =\ 1\,,
\]
 so the inequality~\eqref{Eq:codimension} does not hold.
 Nevertheless, we computed 400,000 instances of this problem using 265 gigaHertz-days (see
 Table~\ref{T:8}),
%
%
\begin{table}[htb]
 \caption[Schubert problem with 8 solutions.]{Schubert problem $\left\{
   \bigl(\,\I,1\bigr)\,,\,
   \bigl(\,\TI,1\bigr)\,,\,
   \bigl(\,\ThTI,2\bigr)
  \right\}$ with 8 solutions}\label{T:8}

 \begin{tabular}{|c||c|c|c|c|c||c|}\hline
  Num. real&0&2&4&6&8&Total\\\hline\hline
  Frequency&160337&0&39663&0&200000&400000\\\hline
 \end{tabular}
\end{table}
observing a congruence modulo four.

Finally, the Schubert problem
 $\blambda=(\ThII, \ThII, \TI,\TI)$
 on $\Gr(4,\C^8)$ has 8 solutions.
 We have
\[
   \|\TI\| + \|\TI\| +  \|\ThII\| +  \|\ThII\| +
    \ -\ \dim\LG(\C^8)
  \ =\
 2+2+3+3-10\ =\ 0\,,
\]
 so the inequality~\eqref{Eq:codimension} does not hold.
 In fact $\blambda$ gives a Schubert problem on $\LG(\C^8)$ with four solutions, so we
 have that $f((X_{\bmu})_\angle)=Z_{\bmu}$.
 Nevertheless, we computed 400,000 instances of this problem using 3.2 gigaHertz-years
 (see  Table~\ref{T:8II}),
%
%
\begin{table}[htb]
 \caption[Schubert problem with 8 solutions.]{Schubert problem $\left\{
   \bigl(\,\TI,2\bigr)\,,\,
   \bigl(\,\ThII,2\bigr)
  \right\}$ with 8 solutions.}\label{T:8II}
 \begin{tabular}{|c||c|c|c|c|c||c|}\hline
  Num. real&0&2&4&6&8&Total\\\hline\hline
  Frequency&147611&0&152389&0&100000&400000\\\hline
 \end{tabular}
\end{table}
observing a congruence modulo four.

We studied fibers of $f\colon X_{\bmu}\to Z_{\bmu}$ over points of $U_{\bmu}(\R)$ for all
symmetric Schubert problems on $\Gr(m,V)$ when $m\leq 4$ whose degree $d(\bmu)$ was
at most $96$, a total of $44$ Schubert problems in all.
For each, we computed the fibers of $X_{\bmu}$ over several hundred thousand points in
$U_{\bmu}(\R)$, determining the number of real points in each fiber.
These data are recorded in frequency tables such as those we have given here.
These are available on line~\cite{Lower_Sym} and are part of a larger
experiment~\cite{Lower_Exp}.
Of these, 21 satisfy the inequality~\eqref{Eq:codimension} and each exhibited a
congruence modulo four.
(Ten satisfy the weaker condition~\eqref{Eq:dimension_estimate}.)
Four of the remaining 23 do not satisfy the inequality~\eqref{Eq:codimension}, but still
had the congruence modulo four, and the remaining 19 neither satisfy the
inequality~\eqref{Eq:codimension}, nor have a congruence modulo four.

%
\subsection{Further Lacunae}\label{S:gaps}

We close with two more symmetric Schubert problems which exhibit additional lacunae in
their observed numbers of real solutions.
The first is the problem
$\{ (\ThThTh,1),
    (\I,7)\}$
on $\Gr(4,\C^8)$ with 20 solutions.
The inequality~\eqref{Eq:simple_condition} holds for this problem, so the possible numbers
of real solutions are congruent to 20 modulo four.
Table~\ref{T:20} displays the result of computing 400,000 instances
\begin{table}[htb]
%
%
\caption[Schubert problem with 20 solutions.]{Schubert problem $\left\{
   \bigl(\,\ThThTh,1\bigr)\,,\,
   \bigl(\,\I,7\bigr)
  \right\}$ with 20 solutions.}\label{T:20}
\[
\begin{tabular}{|c||c|c|c|c|c|c|c|c|c|c|c|c|c|c||c|}\hline
  Num. real&0&2&4&6&8&10&12&14&16&18&20&Total\\\hline\hline
  Frequency&7074&0&114096&0&129829&0&0&0&0&0&119001&400000\\\hline
\end{tabular}
\]
\end{table}
which used 2 gigaHertz-days of computing.
In this computation, we did not observe 12 or 16 real solutions.
This is one of a family of Schubert problems on $\Gr(m,\C^{m{+}p})$ for which there are
provable lower bounds and lacunae.
This is explained in~\cite{Lower}, which describes the larger
experiment~\cite{Lower_Exp}.

Our last example is the symmetric Schubert problem
$\left\{
   \bigl(\,\I,1\bigr)\,,\,
   \bigl(\,\TI,2\bigr)\,,\,
   \bigl(\,\TT,1\bigr)\,,\,
   \bigl(\,\ThII,1\bigr)
  \right\}$
on $\Gr(4,\C^8)$ with 16 solutions.
The inequality~\eqref{Eq:codimension} does not hold, so we expect every even number of
solutions between 0 and 16 to occur.
Table~\ref{T:20} displays the result of computing 200,000 instances
\begin{table}[htb]
%
%
\caption[Schubert problem with 16 solutions.]{Schubert problem $\left\{
   \bigl(\,\I,1\bigr)\,,\,
   \bigl(\,\TI,2\bigr)\,,\,
   \bigl(\,\TT,1\bigr)\,,\,
   \bigl(\,\ThII,1\bigr)
  \right\}$ with 16 solutions}
\[
 \begin{tabular}{|c||c|c|c|c|c|c|c|c|c||c|}\hline
  Num. real&0&2&4&6&8&10&12&14&16&Total\\\hline\hline
  Frequency&37069&16077&24704&10&22140&0&0&0&100000&200000\\\hline
 \end{tabular}
\]
\end{table}
which used $8.4$ gigaHertz-days of computing.
In this computation, we did not observe 10 or 12 or 14 real solutions.
We do not know a reason for this gap in the observed number of real solutions.

\providecommand{\bysame}{\leavevmode\hbox to3em{\hrulefill}\thinspace}
\providecommand{\MR}{\relax\ifhmode\unskip\space\fi MR }
\providecommand{\MRhref}[2]{%
  \href{http://www.ams.org/mathscinet-getitem?mr=#1}{#2}
}
\providecommand{\href}[2]{#2}

\end{document}